\documentclass[a4paper, 10pt, notitlepage]{article}

\usepackage{amsthm, amsmath, amssymb, latexsym}
\usepackage{mathrsfs,cite}
\usepackage[multiple]{footmisc}
\usepackage{graphicx}

\theoremstyle{plain}
\newtheorem{theorem}{Theorem}[section]
\newtheorem{lemma}[theorem]{Lemma}
\newtheorem{corollary}[theorem]{Corollary}
\newtheorem{proposition}[theorem]{Proposition}

\theoremstyle{definition}
\newtheorem{definition}[theorem]{Definition}

\theoremstyle{remark}
\newtheorem{remark}{Remark}[section]

\usepackage{chngcntr}
\counterwithin{theorem}{section}  
\counterwithin{remark}{section}  

\DeclareMathOperator{\dom}{dom}
\DeclareMathOperator{\diag}{diag}
\DeclareMathOperator{\Ker}{Ker}
\DeclareMathOperator{\Image}{Im}

\author{M.V. Dolgopolik\footnote{Institute for Problems in Mechanical Engineering of the Russian Academy of Sciences,
Saint Petersburg, Russia}\footnote{This work was performed in IPME RAS and supported by the Russian Science Foundation
(Grant No. 20-71-10032).}}
\title{Exact augmented Lagrangians for constrained optimization problems in Hilbert spaces {I}: Theory}

\begin{document}

\maketitle

\begin{abstract}
In this two-part study, we develop a general theory of the so-called exact augmented Lagrangians for constrained
optimization problems in Hilbert spaces. In contrast to traditional nonsmooth exact penalty functions, these augmented
Lagrangians are continuously differentiable for smooth problems and do not suffer from the Maratos effect, which makes
them especially appealing for applications in numerical optimization. Our aim is to present a detailed study of various
theoretical properties of exact augmented Lagrangians and discuss several applications of these functions to
constrained variational problems, problems with PDE constraints, and optimal control problems.

The first paper is devoted to a theoretical analysis of an exact augmented Lagrangian for optimization problems in
Hilbert spaces. We obtain several useful estimates of this augmented Lagrangian and its gradient, and present several
types of sufficient conditions for KKT-points of a constrained problem corresponding to locally/globally optimal
solutions to be local/global minimisers of the exact augmented Lagrangian.
\end{abstract}

\section{Introduction}

The concept of \textit{exactness} of a penalty function was first introduced by Eremin \cite{Eremin} and Zangwill
\cite{Zangwill} in the mid-1960s. The penalty function $F_c(x) = f(x) + c \| h(x) \|$ for the constrained optimization
problem 
\[
  \min \: f(x) \quad \text{subject to} \quad h(x) = 0
\]
is called exact, if for any sufficiently large (but finite) value of the penalty parameter $c > 0$ its points of
global/local minimum coincide with globally/locally optimal solutions of the constrained problem. It turned out that
penalty functions for many convex and nonconvex constrained optimization problems are exact under relatively mild
assumptions
\cite{HanMangasarian,DiPilloGrippo_88_ExPen,ExactBarrierFunc,Demyanov,Zaslavski,Zaslavski2009,Zaslavski2013,
Dolgopolik_ExPenFunctions} ,
which made them one of the cornerstones of constrained optimization for several decades. However, exact penalty
functions have several drawbacks. Firstly, they are inherently nonsmooth (see, e.g.
\cite[Remark~3]{Dolgopolik_ExPenFunctions} for details), which means that one either has to develop specific
numerical methods for minimising such functions (see, e.g. \cite{MaratosPHD,Polak_book}) or use them as merely an
auxiliary tool for stepsize evaluation, e.g. in SQP methods \cite{NocedalWright,BonnansGilbert,ConnGouldToint}.
Secondly, numerical methods based on nonsmooth exact penalty functions often suffer from the Maratos effect, which
makes them significantly less appealing for applications (see \cite[Example~15.4]{NocedalWright} and
\cite[Example~17.6]{BonnansGilbert} for simple particular examples of the Maratos effect, and \cite{Coope} for a
discussion of ways to counter this effect).
 
In 1970, Fletcher \cite{Fletcher70} introduced a new penalty function, which overcomes drawbacks of traditional
nonsmooth penalty functions. Namely, under natural assumptions Fletcher's penalty function is exact, continuously
differentiable, and robust with respect to the Maratos effect (see \cite{NocedalWright,ConnGouldToint}). This penalty
function is, in essence, the Hestenes-Powell-Rockafellar augmented Lagrangian 
$\mathscr{L}(x, \lambda, c) = f(x) + \langle \lambda, h(x) \rangle + c \| h(x) \|^2$
(see \cite{BirginMartinez,Rockafellar1974,Rockafellar1993,ItoKunisch}), in which the Lagrange multipliers $\lambda$
are replaced
by their estimates $\lambda(x)$, computed as a solution of a system of linear equations related to the KKT optimality 
conditions. Fletcher's penalty function and its various modifications were studied in details in
\cite{Fletcher73,HanMangasarian_C1PenFunc,DiPilloGrippo85,Lucidi92,ContaldiDiPilloLucidi_93,FukudaSilva}. Despite its
merits, Fletcher's penalty function has been traditionally considered impractical due to the fact that each evaluation
of this function (and especially its gradient) is very computationally expensive. Nevertheless, an efficient
implementation of a constrained optimization method based on Fletcher's penalty function was recently proposed in
\cite{EstrinFriedlander1,EstrinFriedlander2}.
 
In 1979, Di Pillo and Grippo \cite{DiPilloGrippo1979} introduced a new augmented Lagrangian for equality constrained
optimization problems, which can be viewed as a modification of Fletcher's penalty function. Instead of replacing 
the Lagrange multipliers $\lambda$ in the the augmented Lagrangian 
$\mathscr{L}(x, \lambda, c) = f(x) + \langle \lambda, h(x) \rangle + c \| h(x) \|^2$ with their estimates, Di Pillo
and Grippo proposed to add an auxiliary term to this function that penalizes the violation of the KKT optimality
conditions and is directly connected to the system of linear equations for computing the estimate $\lambda(x)$ for
Fletcher's penalty function. In \cite{DiPilloGrippo1979}, it was shown that under some additional assumptions,
local/global minimisers of the Di Pillo-Grippo augmented Lagrangian \textit{jointly in primal and dual variables} are
precisely KKT-points of the equality constrained problem corresponding to its locally/globally optimal solutions. This
analogy with exact penalty functions led to the fact that Di Pillo-Grippo augmented Lagrangians were later on called
\textit{exact}. Furtermore, just like Fletcher's penalty function, Di Pillo-Grippo augmented Lagrangian does not suffer
from the Maratos effect.

The exact augmented Lagrangian from \cite{DiPilloGrippo1979} was extended to the case of inequality constrained
problems in \cite{DiPilloGrippo1982}. The augmented Lagrangian from \cite{DiPilloGrippo1982} was further modified,
analysed, and applied to various inequality constrained optimization problems in
\cite{DiPilloGrippo1982,Lucidi1988,DiPilloLucidi1996,DiPilloLucidi2001,DiPilloEtAl2002,DiPilloLiuzzi2003,LuoWuLiu2013}.
Exact augmented Lagrangians for equality constrained problems were further studied in
\cite{DuZhangGao2006,DuLiangZhang2006}, while such augmented Lagrangians for problems with equality and two-sided
(box) constraints were discussed in \cite{DiPilloLucidiPalagi1993,DiPilloLiuzzi2011}. Numerical methods for
solving constrained optimization problems based on exact augmented Lagrangians were studied in
\cite{DiPilloLucidiPalagi2000,DiPilloLucidi2005,DiPilloLiuzzi}.

Fukuda and Louren\c{c}o \cite{FukudaLourenco} extended the theory of exact augmented Lagrangians to the case of
nonlinear semidefinite programming problems. Finally, a general theory of exact augmented Lagrangians for cone
constrained optimization problems was developed by the author in \cite{Dolgopolik_AugmLagr}. In particular, in
\cite{Dolgopolik_AugmLagr} it was shown that one can construct an exact augmented Lagrangian from many other other
augmented Lagrangians apart from the Hestenes-Powell-Rockafellar augmented Lagrangian. However, to the best of the
author's knowledge, all existing results on exact augmented Lagrangians were obtained only in the finite dimensional
case.

The main of goal of this study is to develop a general theory of exact augmented Lagrangians for optimization problems
in infinite dimensional spaces and to develop new exact augmented Lagrangian methods for solving these problems. 
The motivation behind such extension is based on the fact that exact augmented Lagrangians have been used to develop
efficient superlinearly convergent optimization methods for nonlinear programming problems that are robust with respect
to the Maratos effect (see
\cite{DiPilloLucidi2001,DiPilloEtAl2002,DiPilloLiuzzi2003,DiPilloLucidiPalagi1993,DiPilloLiuzzi2011}). Our goal is to
develop a theory that would allow one to extend these methods to optimization problems in infinite dimensional
spaces, such as optimal control problems and problems with PDE constraints.

The first part of our study is devoted to a theoretical analysis of exact augmented Lagrangians for optimization
problems in Hilbert spaces. We restrict our consideration to the Hilbert space setting, since it is unclear whether
exact augmented Lagrangians can be defined in a more general case of optimization problems in Banach spaces.
Nevertheless, in the second paper we will show that in many particular cases, numerical methods based on exact augmented
Lagrangians work well for problems in Banach spaces, although our theoretical results do not permit such a general
problem setting.

In this paper, we introduce an exact augmented Lagrangian for optimization problems in real Hilbert spaces with
inequality and nonlinear operator equality constraints. We study some properties of this augmented Lagrangian and
obtain some useful estimates of this function and its gradient, which play a crucial role in the analysis of its
exactness. We also obtain sufficient conditions for the exact augmented Lagrangian to have bounded sublevel sets and
study its exact penalty properties with the use of a nonlocal constraint qualification, which is closely related to
the linear independence constraint qualification and conditions on nonlocal metric regularity of constraints. In
particular, we show that under some general assumptions, local/global minimisers of the exact augmented Lagrangian
jointly in primal and dual variables are precisely KKT-points of the original problem corresponding to its
locally/globally optimal solutions, provided the penalty parameter is sufficiently large. Various applications of the
theory developed in this paper to particular classes of constrained variational problems, problems with PDE constraints,
and optimal control problems will be presented in the second part of our study.

The paper is organized as follows. The problem statement and the definition of the exact augmented Lagrangian for
optimization problems in Hilbert spaces are discussed in Section~\ref{sect:Definition}. Various properties of the
augmented Lagrangian are studied in Section~\ref{sect:Properties}. In particular, this section contains sufficient
conditions for sublevel sets of the exact augmented Lagrangian to be bounded. Some useful estimates of the gradient of
the augmented Lagrangian are collected in Section~\ref{sect:GradientProperties}. Finally,
Seciton~\ref{sect:Exactness} is devoted to various sufficient conditions for the local/global exactness of the augmented
Lagrangian introduced in this paper.

\section{The definition of exact augmented Lagrangian}
\label{sect:Definition}

Let $X$ and $H$ be real Hilbert spaces. Throughout this article we study the following constrained optimization
problem:
\[
  \min \: f(x) \quad \text{subject to} \quad F(x) = 0, \quad g_i(x) \le 0, \quad i \in M.
  \eqno{(\mathcal{P})}
\]
Here $f, g_i \colon X \to \mathbb{R}$ and $F \colon X \to H$ are given functions, and $M = \{ 1, \ldots, m \}$. Below, 
we suppose that there exists a globally optimal solution of the problem $(\mathcal{P})$. Our aim is to reduce 
the problem $(\mathcal{P})$ to a completely equivalent \textit{unconstrained} problem of minimising a certain augmented
Lagrangian in primal and dual variables simultaneously. Following Di Pillo, Grippo, and Lucidi
\cite{DiPilloGrippo1979,DiPilloGrippo1982,Lucidi1988,DiPilloLucidi1996,DiPilloLucidi2001}, we call such 
functions \textit{exact augmented Lagrangians}.

Let $\langle \cdot, \cdot \rangle$ be the inner product in $X$, $H$ or $\mathbb{R}^n$, depending on the context, and
$g(\cdot) = (g_1(\cdot), \ldots, g_m(\cdot))$. In the case when the functions $g_i$ are differentiable, we denote
by $\nabla g(x) y \in \mathbb{R}^m$ the vector whose $i$-th coordinate is $\langle \nabla g_i(x), y \rangle$, where 
$y \in X$. Let
\[
  L(x, \lambda, \mu) = f(x) + \langle \lambda, F(x) \rangle + \langle \mu, g(x) \rangle,
  \quad \lambda \in H, \quad \mu \in \mathbb{R}^m,
\] 
be the classical Lagrangian for the problem $(\mathcal{P})$. 

To include several particular cases into a general theory, choose a convex non-decreasing lower semicontinuous
(l.s.c.) function $\phi \colon [0, + \infty) \to [0, + \infty]$ such that $\phi(t) = 0$ if and only if $t = 0$, and 
$\dom \phi \ne \{ 0 \}$. In particular, one can define 
\[
  \phi(t) \equiv t \quad \text{ or } \quad
  \phi(t) = \begin{cases} 
    t / (\alpha - t), & \text{if } t \in [0, \alpha),
    \\
    + \infty, & \text{if } t \ge \alpha,
  \end{cases}
  \quad \text{ or } \quad \phi(t) = e^t - 1
\]
(here $\alpha > 0$ is fixed). From the assumptions on the function $\phi$ (in particular, its convexity) it follows
that $\phi$ is continuous on its effective domain, and either $\dom \phi = [0, + \infty)$ and $\phi(t) \to + \infty$ as
$t \to \infty$ or there exists $\alpha > 0$ such that $\dom \phi = [0, \alpha)$ and $\phi(t) \to + \infty$ as 
$t \to \alpha$.

Choose also a continuously differentiable concave function $\psi \colon [0, + \infty)^m \to \mathbb{R}$ such that
$\psi(0) > 0$, zero is a point of global maximum of $\psi$, and
\[
  \frac{\partial \psi}{\partial y_i} (y_1, \ldots, y_{i - 1}, 0, y_{i + 1}, \ldots, y_m) = 0
  \quad \forall y \in \mathbb{R}^m, \: i \in M.
\]
The equality on partial derivatives ensures that the function $\psi(\max\{ g(x), 0 \})$ is continuously differentiable,
provided the functions $g_i$ are differentiable. Note that one can define
\[
  \psi(y) \equiv 1 \quad \text{or} \quad \psi(y) = \beta - \sum_{i = 1}^m y_i^s
\]
for some $\beta > 0$ and $s > 1$.

Let $|\cdot|$ be the Euclidean norm in $\mathbb{R}^n$. Introduce the functions
\[
  b(x) = \psi\big( \max\{ g(x), 0 \} \big), \quad 
  p(x, \mu) = \frac{b(x)}{1 + |\mu|^2} \quad \forall x \in X, \: \mu \in \mathbb{R}^m.
\]
Denote $\Omega = \{ x \in X \mid b(x) > 0, \phi(\| F(x) \|^2) < + \infty \}$. It should be noted that the set
$\Omega$ is open, provided the functions $F$ and $g_i$ are continuous. For any vectors 
$y, z \in \mathbb{R}^m$, let $\max\{ y, z \} \in \mathbb{R}^m$ be the vector whose $i$-th coordinate 
is $\max\{ y_i, z_i \}$. The vector $\min\{ y, z \}$ is defined in the same way.

Finally, suppose that the functions $f$, $F$, and $g_i$, $i \in M$, are continuously Fr\'{e}chet differentiable and
introduce the following augmented Lagrangian:
\begin{multline} \label{eq:ExactAugmLagr}
  \mathscr{L}(x, \lambda, \mu, c) = f(x) 
  + \langle \lambda, F(x) \rangle  + \frac{c}{2} (1 + \| \lambda \|^2) \phi(\| F(x) \|^2) + 
  \\
  + \left\langle \mu, \max\left\{ g(x), -\frac{1}{c} p(x, \mu) \mu \right\} \right\rangle
  + \frac{c}{2 p(x, \mu)} \left| \max\left\{ g(x), -\frac{1}{c} p(x, \mu) \mu \right\} \right|^2
  \\
  + \eta(x, \lambda, \mu),
\end{multline}
if $x \in \Omega$, and $\mathscr{L}(x, \lambda, \mu, c) = + \infty$, otherwise. Here $\lambda \in H$ and 
$\mu \in \mathbb{R}^m$ are Lagrange multipliers, $c > 0$ is the penalty parameter,
\begin{equation} \label{eq:EtaDef}
\begin{split}
  \eta(x, \lambda, \mu) &= \frac{1}{2} \Big\| D F(x)\big[ \nabla_x L(x, \lambda, \mu) \big] \Big\|^2
  \\
  &+ \frac{1}{2} \sum_{i = 1}^m 
  \Big( \langle \nabla g_i(x), \nabla_x L(x, \lambda, \mu) \rangle + g_i(x)^2 \mu_i \Big)^2,
\end{split}
\end{equation}
and $D F(x)[\cdot] \colon X \to H$ is the Fr\'{e}chet derivative of the nonlinear operator $F$ at $x$, while 
$\nabla_x L(x, \lambda, \mu)$ is the gradient of the function $x \mapsto L(x, \lambda, \mu)$ (recall that $X$ is a
Hilbert space, which implies that the gradient of any real-valued differentiable function on $X$ is correctly defined).
Augmented Lagrangian \eqref{eq:ExactAugmLagr} is a natural extension of the definition of exact augmented Lagrangian
for mathematical programming problems from
\cite{DiPilloGrippo1979,Lucidi1988,DiPilloLucidi1996,DiPilloLucidi2001,DiPilloLiuzzi2003,Dolgopolik_AugmLagr} to 
the infinite dimensional case.

\begin{remark}
As one can readily verify, the following equality holds true:
\begin{multline*}
  \mathscr{L}(x, \lambda, \mu, c) = L(x, \lambda, \mu) 
  + \bigg[ \frac{c}{2} (1 + \| \lambda \|^2) \phi(\| F(x) \|^2) 
  \\
  + \frac{c}{2 p(x, \mu)} \big( |g(x)|^2 - |\min\{ 0, g(x) + c^{-1} p(x, \mu) \mu \}|^2 \big) \bigg]
  + \eta(x, \lambda, \mu),
\end{multline*}
if $x \in \Omega$. Observe that the augmented Lagrangian $\mathscr{L}(x, \lambda, \mu, c)$ consists of three terms.
The first one is just the standard Lagrangian $L(x, \lambda, \mu)$ for the problem $(\mathcal{P})$. The second term, 
roughly speaking, penalizes the violation of the constraints of the problem $(\mathcal{P})$ and resists an excessive 
increase of the norm of the Lagrange multipliers $\lambda$ and $\mu$. Finally, the term $\eta(x, \lambda, \mu)$, 
in a sense, penalizes the violation of the Karush-Kuhn-Tucker (KKT) optimality conditions.
\end{remark}

\section{Properties of the augmented Lagrangian}
\label{sect:Properties}

Let us present some auxiliary results of the augmented Lagrangian $\mathscr{L}(x, \lambda, \mu, c)$, which will be used
in the following sections. First, we point out continuity and differentiability properties of this function, which can
be readily verified directly.

\begin{proposition} \label{prp:Derivatives}
For any $c > 0$ the function $(x, \lambda, \mu) \mapsto \mathscr{L}(x, \lambda, \mu, c)$ is lower semicontinuous on 
$X \times H \times \mathbb{R}^d$ and continuous on its effective domain $\Omega \times H \times \mathbb{R}^d$.
Moreover, this function is continuously Fr\'{e}chet differentiable on $\Omega \times H \times \mathbb{R}^d$, provided
the functions $f$, $F$, and $g_i$, $i \in M$, are twice continuously Fr\'{e}chet differentiable on $\Omega$ and 
the function $\phi$ is continuously differentiable on its effective domain. Under these conditions for any 
$x \in \Omega$, $\lambda \in H$, $\mu \in \mathbb{R}^m$, and $c > 0$ one has
\begin{align*}
  \nabla_x &\mathscr{L}(x, \lambda, \mu, c) = \nabla_x L(x, \lambda, \mu) 
  + c (1 + \| \lambda \|^2) \phi'\big( \| F(x) \|^2 \big) D F(x)^*[F(x)]
  \\
  &+ \frac{c}{p(x, \mu)} \sum_{i = 1}^m \max\left\{ g_i(x), - \frac{1}{c} p(x, \mu) \mu_i \right\} \nabla g_i(x)
  \\
  &- \frac{c}{b(x) p(x, \mu)} \left| \max\left\{ g(x), - \frac{1}{c} p(x, \mu) \mu \right\} \right|^2
  \sum_{i = 1}^m \frac{\partial \psi}{\partial y_i}(\max\{ g(x), 0 \}) \nabla g_i(x)
  \\
  &+ \Big( D^2 F(x)[\nabla_x L(x, \lambda, \mu), \cdot] \Big)^* \Big[ D F(x)[\nabla_x L(x, \lambda, \mu)] \Big]
  \\
  &+ \Big( D_x (\nabla_x L(x, \lambda, \mu)) \Big)^* \Big[ D F(x)^*\big[ D F(x)[\nabla_x L(x, \lambda, \mu)] \big] \Big]
  \\
  &+ \sum_{i = 1}^m \Big( \langle \nabla_x L(x, \lambda, \mu), \nabla g_i(x) \rangle + g_i(x)^2 \mu_i \Big)
  \Big[ D(\nabla g_i(x))^*[\nabla_x L(x, \lambda, \mu)] 
  \\
  &+ \Big( D_x (\nabla_x L(x, \lambda, \mu)) \Big)^* [ \nabla g_i(x)] + 2 g_i(x) \mu_i \nabla g_i(x) \Big],
\end{align*}
and
\begin{align*}
  \nabla_{\lambda} \mathscr{L}(x, \lambda, \mu, c) &= F(x) + c \phi\big( \| F(x) \|^2 \big) \lambda
  \\
  &+ D F(x) \Big[ D F(x)^* \Big( D F(x)\big[ \nabla_x L(x, \lambda, \mu) \big] \Big) \Big]
  \\
  &+ \sum_{i = 1}^m \Big( \langle \nabla_x L(x, \lambda, \mu), \nabla g_i(x) \rangle + g_i(x)^2 \mu_i \Big) 
  D F(x)^*[\nabla g_i(x)],
\end{align*}
and
\begin{align*}
  \nabla_{\mu} \mathscr{L}(x, \lambda, &\mu, c) = \max\left\{ g(x), -\frac{1}{c} p(x, \mu) \mu \right\}
  \\
  &+ \frac{c}{b(x)} \left| \max\left\{ g(x), -\frac{1}{c} p(x, \mu) \mu \right\} \right|^2 \mu
  \\
  &+ \nabla g(x) \Big( D F(x)^*\Big[ D F(x) \big[ \nabla_x L(x, \lambda, \mu) \big] \Big] \Big)
  \\
  &+ \Big( Gr(x) + \diag(g_i(x)^2) \Big) \Big[ \nabla g(x) \nabla_x L(x, \lambda, \mu) + \diag(g_i(x)^2) \mu \Big],
\end{align*}
where $D^2 F(x)[\cdot, \cdot]$ is the second order Fr\'{e}chet derivative of the nonlinear operator $F$, $A^*$ is the
adjoint operator of a bounded linear operator $A$ mapping between Hilbert spaces, $D_x \nabla_x L(x, \lambda, \mu)$ is 
the Fr\'{e}chet derivative of the function $x \mapsto \nabla_x L(x, \lambda, \mu)$, and 
$Gr(x) = \{ \langle \nabla g_i(x), \nabla g_j(x) \rangle \}_{i, j \in M}$ is the Gram matrix of the vectors 
$\nabla g_i(x)$, $i \in M$.
\end{proposition}

\begin{remark}
{(i)~In the case when there are no inequality constraints, the augmented Lagrangian $\mathscr{L}(x, \lambda, c)$ is $k$
times continuously Fr\'{e}chet differentiable in $(x, \lambda)$ on its effective domain for an arbitrary 
$k \in \mathbb{N}$, provided the functions $f$ and $F$ are $k + 1$ times continuously Fr\'{e}chet differentiable, while
the function $\phi$ is $k$ times continuously differentiable on its effective domain.
}

\noindent{(ii)~With the use of the previous proposition one can readily verify that if $(x, \lambda, \mu)$ is a KKT
point of the problem $(\mathcal{P})$ (i.e. $x$ is feasible for this problem, $\nabla_x L(x, \lambda, \mu) = 0$, and for
all $i \in M$ one has $\mu_i g_i(x) = 0$ and $\mu_i \ge 0$), then $\nabla_x \mathscr{L}(x, \lambda, \mu, c) = 0$, 
$\nabla_{\lambda} \mathscr{L}(x, \lambda, \mu, c) = 0$, and $\nabla_{\mu} \mathscr{L}(x, \lambda, \mu, c) = 0$ for all 
$c > 0$. Thus, KKT points of the problem $(\mathcal{P})$ are stationary points of the augmented Lagrangian 
$\mathscr{L}(x, \lambda, \mu, c)$. Below, we will show that under some additional assumptions the converse statement
holds true, that is, stationary points of the augmented Lagrangian are, in fact, KKT points of the problem 
$(\mathcal{P})$.
}
\end{remark}

Let us also obtain a simple, yet useful lower estimate of the augmented Lagrangian.

\begin{lemma} \label{lem:LowerEstimate}
Let there exist $\phi_0 > 0$ such that $\phi(t) \ge \phi_0 t$ for all $t \ge 0$ (or, equivalently, 
$\liminf_{t \to +0} \phi(t) / t > 0$). Then 
\begin{align*}
  \mathscr{L}(x, \lambda, \mu, c) &\ge f(x) + \frac{c}{2} \phi(\| F(x) \|^2) - \frac{1}{2 c \phi_0}
  \\
  &+ \frac{c}{2 b(x)} \left| \max\left\{ g(x), -\frac{1}{c} p(x, \mu) \mu \right\} \right|^2
  - \frac{\psi(0)}{2c} + \eta(x, \lambda, \mu)
  \\
  &\ge f(x) + \frac{c}{2} \phi(\| F(x) \|^2) - \frac{1}{2 c \phi_0}
  \\
  &+ \frac{c}{2 \psi(0)} \big| \max\{ g(x), 0 \} \big|^2
  - \frac{(1 + m)\psi(0)}{2c} + \eta(x, \lambda, \mu)
\end{align*}
for all $x \in X$, $\lambda \in H$, $\mu \in \mathbb{R}^m$, and $c > 0$.
\end{lemma}

\begin{proof}
With the use of the inequality $\phi(t) \ge \phi_0 t$ and the Cauchy-Bunyakovsky-Schwarz inequality one obtains that
\begin{align*}
  \mathscr{L}(x, \lambda, &\mu, c) \ge f(x) + \frac{c}{2} \phi(\| F(x) \|^2)
  - \| \lambda \| \| F(x) \|  + \frac{c \phi_0}{2} \| \lambda \|^2 \| F(x) \|^2 + 
  \\
  &+ \frac{c}{2 b(x)} \left| \max\left\{ g(x), -\frac{p(x, \mu)}{c} \mu \right\} \right|^2
  - |\mu| \left| \max\left\{ g(x), -\frac{p(x, \mu)}{c} \mu \right\} \right|
  \\
  &+ \frac{c |\mu|^2}{2 \psi(0)} \left| \max\left\{ g(x), -\frac{p(x, \mu)}{c} \mu \right\}  \right|^2
  + \eta(x, \lambda, \mu)
\end{align*}
for all $x \in \Omega$, $\lambda \in H$, $\mu \in \mathbb{R}^m$, and $c > 0$. Hence applying the following obvious 
lower estimates
\[
  - t + \frac{c \phi_0}{2} t^2 \ge - \frac{1}{2 c \phi_0}, \quad
  - t + \frac{c}{2 \psi(0)} t^2 \ge - \frac{\psi(0)}{2c} \quad \forall t \in \mathbb{R}, 
\]
one obtains that
\begin{equation} \label{eq:AlmostLowerEstimate}
\begin{split}
  \mathscr{L}(x, \lambda, \mu, c) &\ge f(x) + \frac{c}{2} \phi(\| F(x) \|^2) - \frac{1}{2 c \phi_0}
  \\
  &+ \frac{c}{2 b(x)} \left| \max\left\{ g(x), -\frac{1}{c} p(x, \mu) \mu \right\} \right|^2
  - \frac{\psi(0)}{2c} + \eta(x, \lambda, \mu),
\end{split}
\end{equation}
for all $x \in \Omega$ and $\mu \in \mathbb{R}^m$. 

Denote
\[
  M_+(x, \mu, c) = \left\{ i \in M \Bigm| g_i(x) \ge p(x, \mu)\frac{\mu_i}{c} \right\}, \enspace
  M_{-}(x, \mu, c) = M \setminus M_+(x, \mu, c)
\]
and define $M_+(x) = \{ i \in M \mid g_i(x) \ge 0 \}$, $M_-(x) = M \setminus M_+(x)$. Observe that
\begin{align*}
  \left| \max\left\{ g(x), -\frac{p(x, \mu)}{c} \mu \right\} \right|^2
  &= \sum_{i \in M_+(x, \mu, c)} g_i(x)^2 
  + \sum_{i \in M_-(x, \mu, c)} p(x, \mu)^2 \frac{\mu_i^2}{c^2} 
  \\
  &\ge \sum_{i \in M_+(x, \mu, c)} g_i(x)^2
  = |\max\{ g(x), 0 \}|^2
  \\
  &- \sum_{i \in M_-(x, \mu, c) \cap M_+(x)} g_i(x)^2
  + \sum_{i \in M_+(x, \mu, c) \cap M_-(x)} g_i(x)^2.
\end{align*}
Hence taking into account the fact that $p(x, \mu)|\mu_i| / c \le \psi(0) / c$ for any $i \in M$, that is,
$|g_i(x)| \le \psi(0) / c$ for any $i \in M_-(x, \mu, c) \cap M_+(x)$, one obtains that 
\[
  \left| \max\left\{ g(x), -\frac{1}{c} p(x, \mu) \mu \right\} \right|^2
  \ge |\max\{ g(x), 0 \}|^2 - m \frac{\psi(0)^2}{c^2}.
\]
Combining this estimate and inequality \eqref{eq:AlmostLowerEstimate}, we arrive at the required result in the case 
$x \in \Omega$. The validity of the lemma in the case $x \notin \Omega$ is obvious. 
\end{proof}

Our next goal is to show that under some natural assumptions the function $\mathscr{L}(\cdot, c)$ has bounded sublevel
sets
\begin{equation} \label{def:SublevelSet}
  S_c(\gamma) := \Big\{ (x, \lambda, \mu) \in X \times H \times \mathbb{R}^m \Bigm| 
  \mathscr{L}(x, \lambda, \mu, c) \le \gamma \Big\}, \quad \gamma \in \mathbb{R},
\end{equation}
provided the penalty parameter $c > 0$ is sufficient large and a certain nonlocal (uniform) constraint qualification
holds true. This constraint qualification is reduced to the linear independence constraint qualification (LICQ) in 
the finite dimensional case and plays a key role in the derivation of most results on the augmented Lagrangian
$\mathscr{L}(x, \lambda, \mu, c)$ in this article. Note that in the finite dimensional case LICQ plays a similar role
(cf. \cite{DiPilloGrippo1979,DiPilloGrippo1982,Lucidi1988,DiPilloLucidi1996,DiPilloLucidi2001,DiPilloLiuzzi2003}).

To define the required constraint qualification, introduce the function 
\begin{equation} \label{eq:QuadFormDef}
\begin{split}
  Q(x)[\lambda, \mu] 
  &= \frac{1}{2} \Big\| D F(x)\Big[ D F(x)^*[\lambda] + \sum_{i = 1}^m \mu_i \nabla g_i(x) \Big] \Big\|^2
  \\
  &+ \frac{1}{2} \Big| \nabla g(x) \Big( D F(x)^*[\lambda] + \sum_{i = 1}^m \mu_i \nabla g_i(x) \Big) 
  + \diag(g_i(x)^2) \mu \Big|^2.
\end{split}
\end{equation}
This function is obviously quadratic with respect to $(\lambda, \mu)$. We say that $Q(x)[\cdot]$ is
\textit{positive definite} at a point $x$, if there exists $a > 0$ such that 
\begin{equation} \label{eq:PositiveDefinite}
  Q(x)[\lambda, \mu] \ge a (\| \lambda \|^2 + |\mu|^2)
  \quad \forall \lambda \in H, \: \mu \in \mathbb{R}^m.
\end{equation}
In this case one says that $Q(x)[\cdot]$ is positive definite with constant $a > 0$.

Let us show how the positive definiteness of the function $Q(x)[\cdot]$ is connected with a well-known constraint
qualification for optimization problems in infinite dimensional spaces. For any $x \in X$ denote 
$M(x) = \{ i \in M \mid g_i(x) = 0 \}$, and let $m(x) = |M(x)|$ be the cardinality of the set $M(x)$. For any Hilbert
spaces $Y$ and $Z$, we endow the product space $Y \times Z$ with the inner product 
\[
  \langle (y_1, z_1), (y_2, z_2) \rangle = \langle y_1, y_2 \rangle + \langle z_1, z_2 \rangle
  \quad \forall y_1, y_2 \in Y, \: z_1, z_2 \in Z
\]
and the corresponding norm.  

\begin{lemma} \label{lem:LICQvsPositiveDef}
Let $x \in X$ be fixed. The function $Q(x)[\cdot]$ is positive definite if and only if the linear operator
$\mathcal{T} \colon X \to H \times \mathbb{R}^{m(x)}$, defined as
\[
  \mathcal{T} z = \big\{ D F(x)[z] \big\} \times \prod_{i \in M(x)} \{ \langle \nabla g_i(x), z \rangle \}
  \quad \forall z \in X,
\]
is surjective.
\end{lemma}

\begin{proof}
To simplify the notation, without loss of generality we assume that $M(x) = \{ 1, \ldots, m_0 \}$ for some 
$m_0 \in \mathbb{N}$, $m_0 \le m$.

\textit{Part 1.}~Let $Q(x)[\cdot]$ be positive definite. Then there exists $a > 0$ such that inequality 
\eqref{eq:PositiveDefinite} holds true. Observe that 
\[
  \mathcal{T}^*(\lambda, \nu) = D F(x)^*[\lambda] + \sum_{i \in M(x)} \nu_i \nabla g_i(x)
  \quad \forall \lambda \in H, \: \nu \in \mathbb{R}^{m(x)}. 
\]
Consequently, for any $\lambda \in H$ and $\mu \in \mathbb{R}^m$ such that $\mu_i = 0$ for all $i \notin M(x)$ one has
\[
  \frac{1}{2} \big\| \mathcal{T} \mathcal{T}^*(\lambda, \nu) \big\|^2 = Q(x)[\lambda, \mu]
  \ge a (\| \lambda \|^2 + |\mu|^2),
\]
where the vector $\nu \in \mathbb{R}^{m(x)}$ is obtained from $\mu$ by deleting all those coordinates that
correspond to $i \notin M(x)$. Thus, one has
\[
  \big\| \mathcal{T} \mathcal{T}^*(\lambda, \nu) \big\| 
  \ge \sqrt{2 a} \big\| (\lambda, \nu) \big\|
  \quad \forall (\lambda, \nu) \in H \times \mathbb{R}^{m(x)}.
\]
The operator $\mathcal{T} \mathcal{T}^*$ is obviously self-adjoint. Therefore, by \cite[Thm.~4.13]{Rudin} 
the  inequality above implies that this operator is surjective. Consequently, the operator $\mathcal{T}$ is surjective
as well.

\textit{Part 2.}~Let us prove the converse statement. Suppose that the operator $\mathcal{T}$ is surjective. Define 
linear operator $\mathcal{E} \colon X \times \mathbb{R}^m \to H \times \mathbb{R}^m$ as follows:
\[
  \mathcal{E}\left(\begin{smallmatrix} z \\ \xi \end{smallmatrix}\right) = 
  \left(\begin{smallmatrix} D F(x)[z] \\ \nabla g(x) z + \diag(g_i(x)) \xi \end{smallmatrix}\right)
  \quad \forall z \in X, \: \xi \in \mathbb{R}^m.
\]
It is easily seen that
\[
  \mathcal{E}^*\left(\begin{smallmatrix} \lambda \\ \mu \end{smallmatrix}\right) = 
  \left(\begin{smallmatrix} D F(x)^*[\lambda] + \sum_{i = 1}^m \mu_i \nabla g_i(x) \\ 
  \diag(g_i(x)) \mu \end{smallmatrix}\right)
  \quad \forall \lambda \in H, \: \mu \in \mathbb{R}^m, 
\]
which implies that
\begin{equation} \label{eq:QuadTermViaSelfAdjointOper}
  Q(x)[\lambda, \mu] = \frac{1}{2} \big\| \mathcal{E} \mathcal{E}^*(\lambda, \mu) \big\|^2
  \quad \forall \lambda \in H, \: \mu \in \mathbb{R}^m.
\end{equation}
Hence taking into account \cite[Thm.~4.13]{Rudin} and the fact that the operator $\mathcal{E} \mathcal{E}^*$ is 
obviously self-adjoint, one obtains that the function $Q(x)[\cdot]$ is positive definite if and only if the operator 
$\mathcal{E} \mathcal{E}^*$ is surjective. Let us verify that this operator is surjective, provided the operator 
$\mathcal{E}$ is surjective.

Indeed, let $\mathcal{E}$ be surjective. With the use of the identity
\[
  \langle (\lambda, \mu), \mathcal{E} (z, \xi) \rangle = \langle \mathcal{E}^*(\lambda, \mu), (z, \xi) \rangle
  \quad \forall z \in X, \: \lambda \in H, \: \mu, \xi \in \mathbb{R}^m
\]
one gets that $\Ker \mathcal{E} = (\Image \mathcal{E}^*)^{\bot}$, where $\Ker \mathcal{E}$ is the kernel of 
$\mathcal{E}$ and $(\Image \mathcal{E}^*)^{\bot}$ is the orthogonal complement of the image of $\mathcal{E}^*$. 
Hence bearing in mind the fact that the image $\Image \mathcal{E}^*$ is closed by \cite[Thm.~4.13]{Rudin},
one obtains that $(\Ker \mathcal{E})^{\bot} = \Image \mathcal{E}^*$.

Fix any $(z, \xi) \in H \times \mathbb{R}^{m}$. From the fact that $\mathcal{E}$ is surjective it follows that 
there exists $y \in (\Ker \mathcal{E})^{\bot}= \Image \mathcal{E}^*$ such that $\mathcal{E} y = (z, \xi)$. 
By the definition of image, there exists $(\lambda, \mu) \in H \times \mathbb{R}^m$ such that 
$\mathcal{E}^*(\lambda, \mu) = y$, which implies that $\mathcal{E} \mathcal{E}^*(\lambda, \mu) = (z, \xi)$. 
Thus, the operator $\mathcal{E} \mathcal{E}^*$ is surjective.

To conclude the proof of the lemma, we need to check that the operator $\mathcal{E}$ is surjective, if the operator
$\mathcal{T}$ is surjective. Indeed, fix any $\lambda \in H$ and $\mu \in \mathbb{R}^m$. Our aim is to find $z \in X$ 
and $\xi \in \mathbb{R}^m$ such that $\mathcal{E}(z, \xi) = (\lambda, \mu)$, that is,
\[
  D F(x)[z] = \lambda, \quad \nabla g(x) z + \diag(g_i(x)) \xi = \mu.
\]
Denote by $\nu \in \mathbb{R}^{m(x)}$ the vector obtained from $\mu$ by removing all those coordinates that 
correspond to $i \notin M(x)$. Since the operator $\mathcal{T}$ is surjective by our assumption, one can find 
$z \in X$ such that $\mathcal{T} z = (\lambda, \nu)$, that is,
\[
  D F(x)[z] = \lambda, \quad \langle \nabla g_i(x), z \rangle = \nu_i \quad \forall i \in M(x).
\]
Define $\xi_i = 0$ for any $i \in M(x)$, and put
\[
  \xi_i = \frac{1}{g_i(x)} \Big( \mu_i - \langle \nabla g_i(x), z \rangle \Big) \quad \forall i \notin M(x).
\]
Then by definition $\mathcal{E}(z, \xi) = (\lambda, \mu)$. Thus, the operator $\mathcal{E}$ is surjective, which 
implies that the quadratic function $Q(x)[\cdot]$ is positive definite.
\end{proof}

\begin{corollary} \label{crlr:LICQvsPositiveDef}
Let for any $x \in \Omega$ linear operator $\mathcal{E}(x) \colon X \times \mathbb{R}^m \to H \times \mathbb{R}^m$ be
defined as follows:
\[
  \mathcal{E}(x)\left(\begin{smallmatrix} z \\ \xi \end{smallmatrix}\right) = 
  \left(\begin{smallmatrix} D F(x)[z] \\ \nabla g(x) z + \diag(g_i(x)) \xi \end{smallmatrix}\right)
  \quad \forall z \in X, \: \xi \in \mathbb{R}^m.
\]
For any $x \in \Omega$ the following statements are equivalent:
\begin{enumerate}
\item{the quadratic function $Q(x)[\cdot]$ is positive definite;
\label{stat:PositiveDefinite}}

\item{the operator $\mathcal{E}(x)$ is surjective;
\label{stat:Surjective}}

\item{there exists $a > 0$ such that $\| \mathcal{E}(x) \mathcal{E}(x)^*(\lambda, \mu) \| \ge a \| (\lambda, \mu) \|$
for all $\lambda \in H$ and $\mu \in \mathbb{R}^m$;
\label{stat:TopologicInject}}

\item{the operator $\mathcal{E}(x) \mathcal{E}(x)^*$ is invertible.
\label{stat:Invertible}}
\end{enumerate}
Moreover, the third statement holds true if and only if the function $Q(x)[\cdot]$ is positive definite with constant
$a^2/2$.  
\end{corollary}

\begin{proof}
$\ref{stat:PositiveDefinite} \implies \ref{stat:Surjective}$. If the quadratic function $Q(x)[\cdot]$ is positive 
definite, then by the previous lemma the corresponding operator $\mathcal{T}$ is surjective. Hence arguing in the same
way as in the proof of Lemma~\ref{lem:LICQvsPositiveDef} one obtains that the operator $\mathcal{E}(x)$ is surjective.

$\ref{stat:Surjective} \implies \ref{stat:TopologicInject}$. If the operator $\mathcal{E}(x)$ is surjective, then, 
as was shown in the proof of Lemma~\ref{lem:LICQvsPositiveDef}, the operator $\mathcal{E}(x) \mathcal{E}(x)^*$ is 
surjective as well. Therefore by \cite[Thm.~4.13]{Rudin} the third statement of the corollary holds true.

The implication $\ref{stat:TopologicInject} \implies \ref{stat:Invertible}$ follows directly from
\cite[Thm.~4.13]{Rudin} and the fact that the operator $\mathcal{E}(x) \mathcal{E}(x)^*$ is obviously self-adjoint.

Finally, from the obvious implication $\ref{stat:Invertible} \implies \ref{stat:TopologicInject}$ and equality
\eqref{eq:QuadTermViaSelfAdjointOper} it follows that the last statement of the corollary implies the first one.

It remains to note that the third statement is satisfied if and and only if $Q(x)[\cdot]$ is positive definite with
constant
$a^2 / 2$ due to equality \eqref{eq:QuadTermViaSelfAdjointOper}.
\end{proof}

\begin{remark}
{(i)~Recall that the surjectivity of the Fr\'{e}chet derivative of a nonlinear operator is a central assumption of 
the Lusternik-Graves theorem (see, e.g. \cite{BorweinDontchev}), which by this theorem is equivalent to the metric
regularity of the corresponding operator. In particular, in the context of Lemma~\ref{lem:LICQvsPositiveDef}, the
surjectivity assumption is equivalent to the metric regularity of the mapping 
$W_x(\cdot) = \{ F(\cdot) \} \times \prod_{i \in M(x)} \{ g_i(\cdot) \}$ near the point $x$. Thus, by
Lemma~\ref{lem:LICQvsPositiveDef} the quadratic function $Q(x)[\cdot]$ is positive definite if and only if the mapping
$W_x(\cdot)$ is metrically regular near $x$.
}

\noindent{(ii)~Suppose that the space $H$ is finite dimensional and the constraint $F(x) = 0$ is rewritten as a finite
number of equality constraints $f_j(x) = 0$ for some functions $f_j \colon X \to \mathbb{R}$, 
$j \in \{ 1, \ldots, \ell \}$. Then, as one can readily verify, the operator $\mathcal{T}$ from
Lemma~\ref{lem:LICQvsPositiveDef} is surjective if and only if the gradients $\nabla f_j(x)$, 
$j \in \{ 1, \ldots, \ell \}$, and $\nabla g_i(x)$, $i \in M(x)$, are linearly independent, i.e. LICQ holds true at $x$.
Thus, the positive definiteness of the quadratic function $Q(x)[\cdot]$ is equivalent to the validity of LICQ. Below, we
will use the assumption that the function $Q(x)[\cdot]$ is uniformly positive definite on certain sets, that is, there
exists $a > 0$ such that $Q(x)[\lambda, \mu] \ge a (\| \lambda \|^2 + |\mu|^2)$ for all 
$(\lambda, \mu) \in H \times \mathbb{R}^m$ and for any $x$ from a given set. In the light of this remark, one can
interpret this assumption as nonlocal LICQ or as an assumption on nonlocal metric regularity of the constraints of 
the problem $(\mathcal{P})$. Let us also note that nonlocal CQ and nonlocal metric regularity play a central
role in the theory of exact penalty functions in the infinite dimensional case
\cite{Demyanov,Zaslavski,Zaslavski2009,Zaslavski2013,Dolgopolik_ExPenFunctions,DolgopolikFominyh,Dolgopolik2020}.
}
\end{remark}

For any $\gamma \in \mathbb{R}$ and $c > 0$ introduce the set
\[
  \Omega_c(\gamma) := \Big\{ x \in \Omega \Bigm| 
  f(x) + c \big( \| F(x) \|^2 + |\max\{ g(x), 0 \}|^2 \big) \le \gamma \Big\}.
\]
We are finally ready to obtain sufficient conditions for the boundedness of the sublevel set $S_c(\gamma)$ defined in
\eqref{def:SublevelSet}.

\begin{theorem} \label{thrm:SublevelBoundedness}
Let $\gamma \in \mathbb{R}$ be fixed and the following assumptions be valid:
\begin{enumerate} 
\item{there exist $\phi_0 > 0$ such that $\phi(t) \ge \phi_0 t$ for all $t \ge 0$;}

\item{the set $\Omega_r(\gamma + \varepsilon)$ is bounded for some $r > 0$ and $\varepsilon > 0$;
\label{assumpt:PenaltySublevelBounded}} 

\item{$f$ is bounded below on the set $\Omega_r(\gamma + \varepsilon)$, and $g$ is bounded on this set;}

\item{the gradients $\nabla f(x)$, $\nabla g_i(x)$, $i \in M$, and the Fr\'{e}chet derivative $D F(x)$ are bounded on
the set $\Omega_r(\gamma + \varepsilon)$;}

\item{there exists $a > 0$ such that for all $x \in \Omega_r(\gamma + \varepsilon)$ one has
\begin{equation} \label{eq:NonlocalCQ}
  Q(x)[\lambda, \mu] \ge a \big( \| \lambda \|^2 + |\mu|^2 \big)
  \quad \forall \lambda \in H, \: \mu \in \mathbb{R}^m.
\end{equation}
\label{assumpt:NonlocalCQ}}

\vspace{-5mm}
\end{enumerate}
Then there exists $c_0 > 0$ such that for all $c \ge c_0$ the sublevel set $S_c(\gamma)$ is bounded.
\end{theorem}

\begin{proof}
Note that the function $\mathscr{L}(x, \lambda, \mu, c)$ is non-decreasing in $c$. Therefore, 
$S_{c_1}(\gamma) \subseteq S_{c_2}(\gamma)$ for all $c_1 \ge c_2 > 0$, and it is sufficient to prove that the sublevel
set $S_c(\gamma)$ is bounded only for some $c > 0$.

From Lemma~\ref{lem:LowerEstimate} it follows that for any $(x, \lambda, \mu) \in S_c(\gamma)$ one has
\[
  f(x) + \frac{c \phi_0}{2} \| F(x) \|^2 + \frac{c}{2\psi(0)} \big| \max\{ g(x), 0 \} \big|^2
  \le \gamma + \frac{1}{2 c \phi_0} + \frac{( 1 + m)\psi(0)}{2c}
\] 
Consequently, for any 
\[
  c \ge \widehat{c} := \max\left\{ \frac{2r}{\phi_0}, 2 \psi(0) r, 
  \frac{1}{\varepsilon} \left( \frac{1}{2 \phi_0} + \frac{(1 + m)\psi(0)}{2} \right)  \right\}
\] 
and for all $(x, \lambda, \mu) \in S_c(\gamma)$ one has $x \in \Omega_r(\gamma + \varepsilon)$.

Arguing by reductio ad absurdum, suppose that for any $c > 0$ the set $S_c(\gamma)$ is unbounded. Then for any 
$n \in \mathbb{N}$ there exists $(x_n, \lambda_n, \mu_n) \in S_n(\gamma)$ such that 
$\| x_n \| + \| \lambda_n \| + |\mu_n| \ge n$. As was noted above, for any $n \ge \widehat{c}$
one has $x_n \in \Omega_r(\gamma + \varepsilon)$. Therefore by Assumption~\ref{assumpt:PenaltySublevelBounded} 
the sequence $\{ x_n \}$ is bounded, which implies that $\| \lambda_n \| + |\mu_n| \to + \infty$ as $n \to \infty$.
Moreover,
by Assumption~\ref{assumpt:NonlocalCQ} one has
\begin{equation} \label{eq:LimitingLowerEstimate}
  Q(x_n)[\lambda, \mu] \ge a \big( \| \lambda \|^2 + |\mu|^2 \big)
  \quad \forall \lambda \in H, \: \mu \in \mathbb{R}^m
\end{equation}
for all $n \ge \widehat{c}$

From the definition of the function $\eta$ (see~\eqref{eq:EtaDef}) it follows that this function is quadratic with
respect to $(\lambda, \mu)$ and has the form
\[
  \eta(x, \lambda, \mu) = Q(x)[\lambda, \mu] 
  + \langle Q_{1, \lambda}(x), \lambda \rangle + \langle Q_{1, \mu}(x), \mu \rangle + Q_0(x),
\]
where 
\[
  Q_{1, \lambda}(x) = D F(x) \Big[ D F(x)^* \Big( D F(x)[\nabla f(x)] \Big) \Big]
  + \sum_{i = 1}^m \langle \nabla g_i(x), \nabla f(x) \rangle D F(x)[\nabla g_i(x)],
\]
and $Q_{1, \mu}(x) = (Q_{1, \mu}(x)_1, \ldots, Q_{1, \mu}(x)_m) \in \mathbb{R}^m$ with 
\begin{align*}
  Q_{1, \mu}(x)_i &= \langle D F(x)[\nabla f(x)], D F(x)[\nabla g_i(x)] \rangle 
  \\
  &+ \sum_{j = 1}^m \langle \nabla g_j(x), \nabla f(x) \rangle \langle \nabla g_j(x), \nabla g_i(x) \rangle
  + \langle \nabla g_i(x), \nabla f(x) \rangle g_i(x)^2,
\end{align*}
and
\[
  Q_0(x) = \frac{1}{2} \big\| D F(x)[\nabla f(x)] \big\|^2 
  + \frac{1}{2} \sum_{i = 1}^m \langle \nabla g_i(x), \nabla f(x) \rangle^2.
\] 
From our assumption on the boundedness of the function $g$ and all first order derivatives it follows that there exists
$C > 0$ such that
\[
  \| Q_{1, \lambda}(x_n) \| \le C, \quad |Q_{1, \mu}(x_n)| \le C, \quad Q_0(x_n) \le C \quad \forall n \ge \widehat{c}.
\]
Hence with the use of inequality \eqref{eq:LimitingLowerEstimate} one obtains that
\[
  \eta(x_n, \lambda_n, \mu_n) \ge a \big( \| \lambda_n \|^2 + |\mu_n|^2 \big)
  - C \big( \| \lambda_n \| + |\mu_n| \big) - C \quad \forall n \ge \widehat{c}.
\]
Consequently, applying Lemma~\ref{lem:LowerEstimate} one gets that
\begin{align*}
  \mathscr{L}(x_n, \lambda_n, \mu_n, n) &\ge f(x_n) + \frac{n}{2} \phi(\| F(x_n) \|^2)
  \\
  &+ \frac{n}{2\psi(0)} \left| \max\left\{ g(x_n), 0 \right\} \right|^2
  - \frac{1}{2 n \phi_0} - \frac{1}{2} \left( 1 + m \right) \frac{\psi(0)}{n} 
  \\
  &+ a \big( \| \lambda_n \|^2 + |\mu_n|^2 \big)  - C \big( \| \lambda_n \| + |\mu_n| \big) - C
\end{align*}
for all $n \ge \widehat{c}$. Therefore $\mathscr{L}(x_n, \lambda_n, \mu_n, n) \to + \infty$ as $n \to \infty$, since by
our assumptions $\| \lambda_n \| + |\mu_n| \to + \infty$ as $n \to \infty$, the function $f$ is bounded below on 
the set $\Omega_r(\gamma + \varepsilon)$, and $x_n \in \Omega_r(\gamma + \varepsilon)$ for all $n \ge \widehat{c}$. 
On the other hand, by definition $\mathscr{L}(x_n, \lambda_n, \mu_n, n) \le \gamma$ for all $n \in \mathbb{N}$, which
leads to the obvious contradiction. Thus, there exists $c > 0$ such that the set $S_c(\gamma)$ is bounded, and the proof
is complete.
\end{proof}

\begin{corollary}
Let the following assumptions be valid:
\begin{enumerate} 
\item{there exist $\phi_0 > 0$ such that $\phi(t) \ge \phi_0 t$ for all $t \ge 0$;}

\item{$f$ is bounded below and $g$ is bounded on bounded subsets of $\Omega$;}

\item{the gradients $\nabla f(x)$, $\nabla g_i(x)$, $i \in M$, and the Fr\'{e}chet derivative $D F(x)$ are
bounded on bounded subsets of $\Omega$;}

\item{either the penalty function $\Psi_c(x) = f(x) + c( \| F(x) \|^2 + |\max\{ g(x), 0 \}|^2 )$ is coercive on $\Omega$
for some $c > 0$ or the set $\Omega$ is bounded;
\label{assumpt:PenaltyFuncCoercivity}} 

\item{for any bounded set $V \subset \Omega$ there exists $a > 0$ such that for all $x \in V$ one has
\[
  Q(x)[\lambda, \mu] \ge a \big( \| \lambda \|^2 + |\mu|^2 \big)
  \quad \forall \lambda \in H, \: \mu \in \mathbb{R}^m.
\]
}

\vspace{-5mm}
\end{enumerate}
Then for any $\gamma \in \mathbb{R}$ there exists $c(\gamma) > 0$ such that for all $c \ge c(\gamma)$ the sublevel set 
$S_c(\gamma)$ is bounded.
\end{corollary}

\begin{proof}
Fix any $\gamma \in \mathbb{R}$. From Assumption~\ref{assumpt:PenaltyFuncCoercivity} it follows that for any
$\varepsilon >0$
the set 
\[
  \Omega_c(\gamma + \varepsilon) = \Big\{ x \in \Omega \Bigm| \Psi_c(x) \le \gamma + \varepsilon \Big\}
\]
is bounded. Hence applying Theorem~\ref{thrm:SublevelBoundedness}, we arrive at the required result.
\end{proof}

\section{Properties of the Gradient of $\mathscr{L}(x, \lambda, \mu, c)$}
\label{sect:GradientProperties}

In this section, we prove an auxiliary result, describing an important property of the gradient of the augmented
Lagrangian $\mathscr{L}(x, \lambda, \mu, c)$. Namely, our aim is to show that the norm of the gradient of the function 
$(x, \lambda, \mu) \mapsto \mathscr{L}(x, \lambda, \mu, c)$, denoted by $\nabla \mathscr{L}(x, \lambda, \mu, c)$,
can be estimated from below via the infeasibility measure
\[
  \| F(x) \| + \left|\max\left\{ g(x), - \frac{p(x, \mu)}{c} \mu \right\}\right|,
\]
provided the nonlocal constraint qualification from Theorem~\ref{thrm:SublevelBoundedness} holds true. With the use of 
such estimate one can readily verify that critical points of the augmented Lagrangian are, in fact, KKT-points of
the original problem. In other words, this estimate is instrumental in the proof of the exactness of 
$\mathscr{L}(x, \lambda, \mu, c)$. Moreover, it plays an important role in the design and analysis of numerical methods
based on the use of this augmented Lagrangian (cf. \cite{DiPilloLucidi2001}).

\begin{theorem} \label{thrm:GradientEstimates}
Let the following assumptions be valid:
\begin{enumerate}
\item{$f$, $g_i$, $i \in M$, and $F$ are twice continuously Fr\'{e}chet differentiable on $\Omega$,
$\phi$ is continuously differentiable on its effective domain, and $\phi'(0) > 0$;
}

\item{for some bounded set $V \subseteq \Omega$ there exists $a > 0$ such that for all $x \in V$ one has
\[
  Q[x](\lambda, \mu) \ge a \big( \| \lambda \|^2 + |\mu| \big)^2 \quad \forall \lambda \in H, \: \mu \in \mathbb{R}^m.
\]
}

\vspace{-5mm}

\item{the functions $f$, $g_i$, $i \in M$, and $F$, as well as their first and second order Fr\'{e}chet derivatives, are
bounded on $V$.
}
\end{enumerate}
Then for all $K > 0$ and $\gamma \in \mathbb{R}$, and any bounded set $\Lambda \subset H \times \mathbb{R}^m$ there
exists $c_* > 0$ such that for all $c \ge c_*$ and $(x, \lambda, \mu) \in (V \times \Lambda) \cap S_c(\gamma)$ the
following inequality holds true:
\begin{equation} \label{eq:GradEstim_InfesasibilityMes}
  \big\| \nabla \mathscr{L}(x, \lambda, \mu, c) \big\| 
  \ge K \left( \|F(x)\| + \left|\max\left\{ g(x), - \frac{p(x, \mu)}{c} \mu \right\}\right| \right).
\end{equation}
In particular, if the assumptions of the theorem are satisfied for $V = \Omega_r(\gamma + \varepsilon)$ with some
$r > 0$, $\varepsilon > 0$, and $\gamma \in \mathbb{R}$, then for all $K > 0$ there exists $c_* > 0$ such that 
inequality \eqref{eq:GradEstim_InfesasibilityMes} holds true for all $c \ge c_*$ and 
$(x, \lambda, \mu) \in S_c(\gamma)$.
\end{theorem}

We divide the proof of this theorem into three lemmas. We start with a somewhat cumbersome technical lemma, which is 
the core part of the proof of Theorem~\ref{thrm:GradientEstimates}.

\begin{lemma} \label{lemma:SepGradientEstimates}
Under the assumptions of Theorem~\ref{thrm:GradientEstimates} for all $K > 0$ and $\gamma \in \mathbb{R}$, and any
bounded set $\Lambda \subset H \times \mathbb{R}^m$ there exist $\varkappa > 0$ and $c_* > 0$ such that for any 
$c \ge c_*$ and all $\xi = (x, \lambda, \mu) \in (V \times \Lambda) \cap S_c(\gamma)$ satisfying the inequality 
\begin{equation} \label{eq:MultipliersGradEstimate}
  \Big\| \big( \nabla_{\lambda} \mathscr{L}(\xi, c), 
  \nabla_{\mu} \mathscr{L}(\xi, c) \big) \Big\|
  \le K \left( \big\| F(x) \big\| + \left|\max\left\{ g(x), - \frac{p(x, \mu)}{c} \mu \right\}\right| \right)
\end{equation}
one has
\[ 
  \big\| \nabla_x \mathscr{L}(x, \lambda, \mu, c) \big\| 
  \ge c \varkappa \left( \big\| F(x) \big\| 
  + \left|\max\left\{ g(x), - \frac{p(x, \mu)}{c} \mu \right\}\right| \right).
\]
\end{lemma}

\begin{proof}
Fix any $K > 0$, $\gamma \in \mathbb{R}$, and bounded set $\Lambda \subset H \times \mathbb{R}^m$. By
Corollary~\ref{crlr:LICQvsPositiveDef} and the second assumption of Theorem~\ref{thrm:GradientEstimates}, for any 
$x \in V$ one has
\begin{equation} \label{eq:UniformTopologicalInject}
  \big\| \mathcal{E}(x) \mathcal{E}(x)^*(\lambda, \mu) \big\| \ge \sqrt{2a} \| (\lambda, \mu) \|
  \quad \forall \lambda \in H, \: \mu \in \mathbb{R}^m. 
\end{equation}
With the use of Proposition~\ref{prp:Derivatives} and the definition of $\mathcal{E}(x)$ from 
Corollary~\ref{crlr:LICQvsPositiveDef} one can readily verify that
\begin{align} \notag
  \begin{pmatrix} 
    \nabla_{\lambda} \mathscr{L}(\xi, c) \\ \nabla_{\mu} \mathscr{L}(\xi, c)
  \end{pmatrix}
  &= \begin{pmatrix}
    F(x) + c \phi(\| F(x) \|^2) \lambda \\
    \max\left\{ g(x), - \frac{p(x, \mu)}{c} \mu \right\} 
    + \frac{c}{b(x)} \left| \max\left\{ g(x), - \frac{p(x, \mu)}{c} \mu \right\} \right|^2 \mu
  \end{pmatrix}
  \\
  &+ \mathcal{E}(x) \mathcal{E}(x)^* 
  \begin{pmatrix}
    D F(x)[\nabla_x L(x, \lambda, \mu)] \\ \nabla g(x) \nabla_x L(x, \lambda, \mu) + \diag(g_i(x)^2) \mu
  \end{pmatrix}
  \label{eq:MultipliersDerivExpr}
\end{align}
for any $\xi = (x, \lambda, \mu) \in \Omega \times H \times \mathbb{R}^m$ and $c > 0$. Recall that by our
assumption the function $\phi$ is continuously differentiable, $\phi(0) = 0$, and the nonlinear operator $F$ is bounded
on $V$. Consequently, there exists $\phi_{\max} > 0$ such that for all $x \in V$ one has 
$\phi(\| F(x) \|^2) \le \phi_{\max} \| F(x) \|^2$. Hence applying \eqref{eq:MultipliersDerivExpr},
\eqref{eq:UniformTopologicalInject}, and the definition of $\eta$ (see \eqref{eq:EtaDef}), one obtains that
\begin{multline} \label{eq:EtaUpperEstimate}
  \sqrt{2\eta(x, \lambda, \mu)} 
  \le \frac{1}{\sqrt{2a}} \bigg[ \Big( K + 1 + c \phi_{\max} \| F(x) \| \| \lambda \| \Big) \| F(x) \| 
  \\
  + \bigg( K + 1 + \frac{c |\mu|}{b(x)} \left|\max\left\{ g(x), - \frac{p(x, \mu)}{c} \mu \right\}\right| \bigg) 
  \left|\max\left\{ g(x), - \frac{p(x, \mu)}{c} \mu \right\}\right| \bigg]
\end{multline}
for any $x \in V$, $\lambda \in H$, $\mu \in \mathbb{R}^m$, and $c > 0$ satisfying inequality
\eqref{eq:MultipliersGradEstimate}.

Let us now estimate the norm of $\nabla_x \mathscr{L}(x, \lambda, \mu, c)$. To this end, fix any 
$\xi = (x, \lambda, \mu) \in \Omega \times H \times \mathbb{R}^m$ and $c > 0$, and consider the functions
\[
  D F(x)[\nabla_x \mathscr{L}(x, \lambda, \mu, c)], \quad
  \nabla g(x)[\nabla_x \mathscr{L}(x, \lambda, \mu, c)].
\]
Applying Proposition~\ref{prp:Derivatives}, and adding and subtracting 
\[
  w(x, \mu, c) := 
  \diag(g_i(x)^2) \left( \mu + \frac{c}{p(x, \mu)} \max\left\{ g(x), - \frac{1}{c} p(x, \mu) \mu \right\} \right) 
\]
in the second row, one obtains that
\begin{multline} 
  \begin{pmatrix} 
    D F(x)\big[ \nabla_x \mathscr{L}(\xi, c) \big] 
    \\ 
    \nabla g(x) \big[ \nabla_x \mathscr{L}(\xi, c) \big]
  \end{pmatrix}
  = \begin{pmatrix} 
    D F(x)[\nabla_x L(\xi)] 
    \\ 
    \nabla g(x)[\nabla_x L(\xi)] + \diag(g_i(x)^2) \mu 
  \end{pmatrix}
  \\
  + c \mathcal{E}(x) \mathcal{E}(x)^*
  \begin{pmatrix} 
    (1 + \| \lambda \|^2) \phi'(\| F(x) \|^2) F(x) 
    \\
    \frac{1}{p(x, \mu)} \max\left\{ g(x), - \frac{p(x, \mu)}{c} \mu \right\}
  \end{pmatrix}
  \\
  - \begin{pmatrix} 0 \\ w(x, \mu, c) \end{pmatrix}
  + c \left| \max\left\{ g(x), - \frac{p(x, \mu)}{c} \mu \right\} \right|^2 A_0(x, \mu)	
  \\
  + A_1(\xi) \Big[ D F(x)[\nabla_x L(\xi)] \Big]
  + A_2(\xi) \Big[ \nabla g(x)[\nabla_x L(\xi)] + \diag(g_i(x)^2) \mu \Big],
  \label{eq:LongDerivativeExpr}
\end{multline}
where the vector $A_0(x, \mu)$, and the linear operators $A_1(x, \lambda, \mu)$ and $A_2(x, \lambda, \mu)$ are defined
via the vectors $\lambda$ and $\mu$, the functions $f$, $F$, and $g_i$, $i \in M$, as well as their first and second
order derivatives, and the functions $\psi(\max\{ 0, g(x) \})$ and $\nabla\psi(\max\{ g(x), 0 \})$.

Let $q(x, \mu, c) = \max\{ g(x), - c^{-1} p(x, \mu) \mu \}$ As was pointed out in \cite{DiPilloLucidi2001} (see equality
$(A.3)$), the following equality holds true:
\begin{multline*}
  \diag(g_i(x)) \mu = \diag(\mu_i) q(x, \mu, c)
  \\
  + \frac{c}{p(x, \mu)} \left( \diag\left(\max\left\{ g_i(x), - \frac{p(x, \mu)}{c} \mu_i \right\}\right) -
  \diag(g_i(x)) \right) 
  q(x, \mu, c)
\end{multline*}
(the validity of this equality can be easily verified by considering two cases: $g_i(x) \ge c^{-1} p(x, \mu) \mu_i$
and $g_i(x) < c^{-1} p(x, \mu) \mu_i$). Therefore
\begin{align*}
  w(x, \mu, c) &= \diag(g_i(x)) \bigg[ \diag(\mu_i) q(x, \mu, c) 
  \\
  &+ \frac{c}{p(x, \mu)} \diag\left(\max\left\{ g_i(x), - \frac{1}{c} p(x, \mu) \mu_i \right\}\right)
  q(x, \mu, c) \bigg].
\end{align*}
Hence with the use of \eqref{eq:LongDerivativeExpr} and \eqref{eq:UniformTopologicalInject} one gets that
\begin{align*}
  \Big( \| D F(x) \| &+ \| \nabla g(x) \| \Big) \big\| \nabla_x \mathscr{L}(x, \lambda, \mu, c) \big\|
  \ge - \sqrt{2 \eta(x, \lambda, \mu)} 
  \\
  &+ c \sqrt{2a} 
  \Big[ (1 + \| \lambda \|^2) \phi'(\| F(x) \|^2) \| F(x) \| 
  + \frac{1}{p(x, \mu)} \left| q(x, \mu, c) \right| \Big]
  \\
  &- |\mu| |g(x)| \left| q(x, \mu, c) \right|
  - c \big( \frac{|g(x)|}{p(x, \mu)} + \| A_0(x, \mu) \| \big) 
  \big| q(x, \mu, c) \big|^2 
  \\
  &- \| A_1(x, \lambda, \mu) \| \Big\| D F(x)[\nabla_x L(x, \lambda, \mu)] \Big\|
  \\
  &- \| A_2(x, \lambda, \mu) \| \Big\| \nabla g(x)[\nabla_x L(x, \lambda, \mu)] + \diag(g_i(x)^2) \mu \Big\| 
\end{align*}
for all $x \in V$, $\lambda \in H$, $\mu \in \mathbb{R}^m$, and $c > 0$.

By our assumptions the functions $f$, $g_i$, $i \in M$, and $F$, as well as their first and second order derivatives
are bounded on the set $V$. Consequently, one can find $S_0, S_1, S_2, S_D, S_{\mu}, S_g > 0$ such that
for any $x \in V$ and for any $(\lambda, \mu)$ from the bounded set $\Lambda \subset H \times \mathbb{R}^m$ one has
\begin{gather*}
  \| A_0(x, \mu) \| \le S_0, \quad \| A_1(x, \lambda, \mu) \| \le S_1, \quad \| A_2(x, \lambda, \mu) \| \le S_2 
  \\
  \| D F(x) \| + \| \nabla g(x) \| \le S_D, \quad |\mu| \le S_{\mu}, \quad |g(x)| \le S_g.
\end{gather*}
Moreover, from the convexity of $\phi$ it follows that $\phi'(\| F(x) \|^2) \ge \phi'(0)$ for any $x \in X$.
Hence for any $(x, \lambda, \mu) \in V \times \Lambda$ and for all $c > 0$ one has
\begin{align*}
  S_D \big\| \nabla_x \mathscr{L}(x, \lambda, \mu, c) \big\| 
  &\ge \frac{c \sqrt{2a}}{\max\{ \psi(0), 1/\phi'(0) \}}
  \left( \| F(x) \|  + \big| q(x, \mu, c) \big| \right)
  \\
  &- S_{\mu} S_g \big| q(x, \mu, c) \big|
  - c \left( \frac{S_g(1 + S_{\mu}^2)}{b(x)} + S_0 \right) 
  \big| q(x, \mu, c) \big|^2
  \\
  &- (1 + S_1 + S_2) \sqrt{2\eta(x, \lambda, \mu)},
\end{align*}
which with the use of \eqref{eq:EtaUpperEstimate} implies that
\[
  \big\| \nabla_x \mathscr{L}(x, \lambda, \mu, c) \big\| 
  \ge \frac{c}{S_D} t(x, \mu, c) 
  \Big( \| F(x) \| + \left| \max\left\{ g(x), - \frac{1}{c} p(x, \mu) \mu \right\} \right| \Big),
\]
where
\begin{multline*}
  t(x, \mu, \lambda, c) = \frac{\sqrt{2a}}{\max\{ \psi(0), 1/\phi'(0) \}} - \frac{S_{\mu} S_g}{c} 
  - \left( \frac{S_g(1 + S_{\mu}^2)}{b(x)} + S_0 \right) 
  \big| q(x, \mu, c) \big|
  \\
  - \frac{(1 + S_1 + S_2)}{\sqrt{2a}} \left( \frac{(K + 1)}{c} + \phi_{\max} \| F(x) \| \| \lambda \|
  + \frac{S_{\mu}}{b(x)} \big| q(x, \mu, c) \big| \right). 
\end{multline*}
Let us check that there exist $c_* > 0$ and $t_0 > 0$ such that for all $c \ge c_*$ and
$(x, \lambda, \mu) \in (V \times \Lambda) \cap S_c(\gamma)$ one has $t(x, \mu, \lambda, c) \ge t_0$.
Then putting $\varkappa = t_0 / S_D$ one obtains the required result.

To prove the existence of $c_*$ and $t_0$, it is sufficient to show that for any $\varepsilon > 0$ there exists
$\widehat{c} > 0$ such that for all $c \ge \widehat{c}$ and $(x, \lambda, \mu) \in S_c(\gamma)$ with $x \in V$ the
following inequalities hold true:
\begin{equation} \label{eq:SmallConstraintViol}
  \| F(x) \| < \varepsilon, \quad 
  \left| \max\left\{ g(x), - \frac{1}{c} p(x, \mu) \mu \right\} \right| < \varepsilon, \quad
  b(x) \ge \frac{\psi(0)}{2}.
\end{equation}
Let us prove the existence of such $\widehat{c}$.

Fix any $\varepsilon > 0$. By our assumption the function $f$ is bounded on $V$. Consequently, by
Lemma~\ref{lem:LowerEstimate} for any $c > 0$ and $(x, \lambda, \mu) \in S_c(\gamma)$ with $x \in V$ one has
\begin{align*}
  \gamma \ge \mathscr{L}(x, \lambda, \mu, c) &\ge f_0 + \frac{c \phi'(0)}{2} \| F(x) \|^2 
  + \frac{c}{2 \psi(0)} \left| \max\left\{ g(x), - \frac{1}{c} p(x, \mu) \mu \right\} \right|^2
  \\
  &- \frac{1}{2c} \left( \frac{1}{\phi'(0)} + \psi(0) \right),
\end{align*}
where $f_0 = \inf\{ f(x) \mid x \in V \}$. Therefore, for any
\[
  c \ge \max\left\{ \frac{1}{2} \left( \frac{1}{\phi'(0)} + \psi(0) \right), 
  \frac{2(\gamma + 1 - f_0)}{\varepsilon^2 \phi'(0)}, \frac{2(\gamma + 1 - f_0) \psi(0)}{\varepsilon^2} \right\}
\]
and for all $(x, \lambda, \mu) \in S_c(\gamma)$ with $x \in V$ the first two inequalities in
\eqref{eq:SmallConstraintViol} hold true. 

Note that from the second inequality in \eqref{eq:SmallConstraintViol} it follows that 
$\max\{ g_i(x), 0 \} < \varepsilon$ for all $i \in M$. Consequently, decreasing $\varepsilon > 0$, if necessary, one
can suppose that $b(x) := \psi(\max\{ g(x), 0 \}) \ge \psi(0) / 2$, since by definition zero is a point of global
maximum of the function $\psi$.
\end{proof}

\begin{lemma}
Under the assumptions of Theorem~\ref{thrm:GradientEstimates} for all $K > 0$ and $\gamma > 0$, and any bounded set
$\Lambda \subset H \times \mathbb{R}^m$ one can find $c_* > 0$ such that inequality
\eqref{eq:GradEstim_InfesasibilityMes} is satisfied for all $c \ge c_*$ and 
$(x, \lambda, \mu) \in (V \times \Lambda) \cap S_c(\gamma)$.
\end{lemma}

\begin{proof}
Arguing by reductio ad absurdum, suppose that there exist $K > 0$, $\gamma \in \mathbb{R}$, and a bounded set 
$\Lambda \subset H \times \mathbb{R}^m$ such that for any $c > 0$ one can find 
$\xi_c = (x_c, \lambda_c, \mu_c) \in (V \times \Lambda) \cap S_c(\gamma)$ satisfying the inequality 
\[
  \big\| \nabla \mathscr{L}(\xi_c, c) \big\| 
  < K_c := K \Big( \|F(x_c)\| + \left|\max\left\{ g(x_c), - \frac{1}{c} p(x_c, \mu_c) \mu_c \right\}\right| \Big)
\] 
Then for any $c > 0$ one has
\begin{equation} \label{eq:GradSmallerThanInfeasibilityMes}
  \big\| \nabla_x \mathscr{L}(\xi_c, c) \big\| < K_c, \quad 
  \Big\| \big( \nabla_{\lambda} \mathscr{L}(\xi_c, c), 
  \nabla_{\mu} \mathscr{L}(\xi_c, c) \big) \Big\| < K_c.
\end{equation}
From the second inequality and Lemma~\ref{lemma:SepGradientEstimates} it follows that for any sufficiently large
$c > 0$ one has
\[
  \| \nabla_x \mathscr{L}(\xi_c, c) \| \ge c \varkappa
  \Big( \|F(x_c)\| + \left|\max\left\{ g(x_c), - \frac{1}{c} p(x_c, \mu_c) \mu_c \right\}\right| \Big)
\]
for some $\varkappa > 0$ independent of $c$. However, for any $c > K / \varkappa$ this inequality contradicts
the first inequality in \eqref{eq:GradSmallerThanInfeasibilityMes}. Thus, the statement of
Theorem~\ref{thrm:GradientEstimates} is true. 
\end{proof}

\begin{lemma}
Let the assumptions of Theorem~\ref{thrm:GradientEstimates} be satisfied for $V = \Omega_r(\gamma + \varepsilon)$ with
some $r > 0$, $\varepsilon > 0$, and $\gamma \in \mathbb{R}$. Then for all $K > 0$ there exists $c_* > 0$ such that 
inequality \eqref{eq:GradEstim_InfesasibilityMes} holds true for all $c \ge c_*$ and 
$(x, \lambda, \mu) \in S_c(\gamma)$.
\end{lemma}

\begin{proof}
If the assumptions of Theorem~\ref{thrm:GradientEstimates} are satisfied for $V = \Omega_r(\gamma + \varepsilon)$, then
by Theorem~\ref{thrm:SublevelBoundedness} there exists $c_0 > 0$ such that the set $S_c(\gamma)$ is bounded for all 
$c \ge c_0$. Moreover, as was shown in the proof of Theorem~\ref{thrm:SublevelBoundedness}, in this case there exists
$\widehat{c} > 0$ such that for all $c \ge \widehat{c}$ and $(x, \lambda, \mu) \in S_c(\gamma)$ one has
$x \in \Omega_r(\gamma + \varepsilon)$. Therefore, one can find a bounded set $\Lambda \subset H \times \mathbb{R}^m$
such that $S_c(\gamma) \subseteq V \times \Lambda$ for all $c \ge \max\{ c_0, \widehat{c} \}$, which by the previous
lemma implies that for all $K > 0$ there exists $c_* > 0$ such that inequality \eqref{eq:GradEstim_InfesasibilityMes}
holds true for all $c \ge \max\{ c_*, c_0, \widehat{c} \}$ and $(x, \lambda, \mu) \in S_c(\gamma)$.
\end{proof}

\section{Exactness of the augmented Lagrangian}
\label{sect:Exactness}

This section is devoted to an analysis of several concepts of exactness of the augmented Lagrangian 
$\mathscr{L}(x, \lambda, \mu, c)$. Namely, we present various types of sufficient conditions for this augmented
Lagrangian to be locally, globally or completely exact. These conditions are based either on
the nonlocal constraint qualification introduced in the previous sections and the use of the gradient estimate from
Theorem~\ref{thrm:GradientEstimates} or second order sufficient optimality conditions.

\subsection{Global exactness}

To give a precise definition of what is meant by ``exactness'' of the augmented Lagrangian 
$\mathscr{L}(x, \lambda, \mu, c)$, consider the following auxiliary unconstrained optimization problem:
\begin{equation} \label{eq:AuxiliaryProblem}
  \min_{(x, \lambda, \mu)} \: \mathscr{L}(x, \lambda, \mu, c).
\end{equation}
Under the assumptions of Theorem~\ref{thrm:SublevelBoundedness}, the sublevel set $S_c(\gamma)$ of the augmented
Lagrangian is bounded, which implies that auxiliary problem \eqref{eq:AuxiliaryProblem} has globally optimal solutions,
provided the function $(x, \lambda, \mu) \mapsto \mathscr{L}(x, \lambda, \mu, c)$ is weakly sequentially lower
semicontinuous and $c > 0$ is sufficiently large. We would like to know how these optimal solutions are connected with
globally optimal solutions of the original problem $(\mathcal{P})$. 

Suppose that for any globally optimal solution $x_*$ of the problem $(\mathcal{P})$ there exist 
$\lambda_* \in H$ and $\mu_* \in \mathbb{R}^m$ such that the triplet $(x_*, \lambda_*, \mu_*)$ satisfies 
the KKT optimality conditions:
\[
  \nabla_x L(x_*, \lambda_*, \mu_*) = 0, \quad F(x_*) = 0, \quad \max\{ g(x_*), -\mu_* \} = 0.
\]
Any triplet $(x_*, \lambda_*, \mu_*)$ satisfying these equalities is called a \textit{KKT-point} 
of the problem $(\mathcal{P})$.

\begin{definition}
One says that the augmented Lagrangian $\mathscr{L}(x, \lambda, \mu, c)$ is \textit{globally exact}, if there exists 
$c_* > 0$ such that for all $c \ge c_*$ a triplet $(x_*, \lambda_*, \mu_*)$ is a globally optimal solution of problem
\eqref{eq:AuxiliaryProblem} if and only if $x_*$ is a globally optimal solution of the problem $(\mathcal{P})$ and
$(x_*, \lambda_*, \mu_*)$ is a KKT-point of this problem.
\end{definition}

Thus, if the augmented Lagrangian $\mathscr{L}(x, \lambda, \mu, c)$ is globally exact, then globally optimal solutions
of auxiliary problem \eqref{eq:AuxiliaryProblem} with a sufficiently large value of the penalty parameter $c$ are
precisely KKT-points corresponding to globally optimal solutions of the problem $(\mathcal{P})$. Furthermore, observe
that if $(x_*, \lambda_*, \mu_*)$ is a KKT-point of the problem $(\mathcal{P})$, then 
$\mathscr{L}(x_*, \lambda_*, \mu_*, c) = f(x_*)$ for all $c > 0$ (see \eqref{eq:ExactAugmLagr}). Therefore, if 
the augmented Lagrangian $\mathscr{L}(x, \lambda, \mu, c)$ is globally exact, then for any $c > 0$ large enough optimal
value of problem \eqref{eq:AuxiliaryProblem} coincides with the optimal value of the problem $(\mathcal{P})$, which 
we denote by $f_*$. Recall that by our assumption there exists a globally optimal solution of 
the problem $(\mathcal{P})$, which implies that $f_*$ is finite.

Note that in the general case
\begin{equation} \label{eq:WeakDuality}
  \inf_{(x, \lambda, \mu)} \mathscr{L}(x, \lambda, \mu, c) \le f_* \quad \forall c > 0,
\end{equation}
since $\mathscr{L}(x_*, \lambda_*, \mu_*, c) = f_*$ for any globally optimal solution $x_*$ of the problem 
$(\mathcal{P})$ and the corresponding Lagrange multipliers $\lambda_*$ and $\mu_*$. Let us show that this inequality
turns into equality precisely when the augmented Lagrangian $\mathscr{L}(x, \lambda, \mu, c)$ is globally exact.

\begin{lemma} \label{lem:GlobalExact_OptimalValue}
Let $Q(x_*)[\cdot]$ be positive definite at every globally optimal solution $x_*$ of the problem $(\mathcal{P})$. Then 
the augmented Lagrangian $\mathscr{L}(x, \lambda, \mu, c)$ is globally exact if and only if the optimal value of 
problem \eqref{eq:AuxiliaryProblem} is equal to $f_*$ for some $c > 0$.
\end{lemma}

\begin{proof}
Suppose that the optimal value of problem \eqref{eq:AuxiliaryProblem} is equal to $f_*$ for some $c > 0$. Our aim is to
show that for any $r > c$ points of global minimum of the augmented Lagrangian $\mathscr{L}(\cdot, r)$ are precisely
KKT-points of the problem $(\mathcal{P})$ corresponding to its globally optimal solutions. Then by definition 
$\mathscr{L}(x, \lambda, \mu, c)$ is globally exact.

Fix any $r > c$. From \eqref{eq:WeakDuality} and the fact that the function $\mathscr{L}(x, \lambda, \mu, c)$ is
non-decreasing in $c$ it follows that $\inf_{(x, \lambda, \mu)} \mathscr{L}(x, \lambda, \mu, r) = f_*$. Consequently,
for any globally optimal solution $x_*$ of the problem $(\mathcal{P})$ and the corresponding Lagrange multipliers
$\lambda_*$ and $\mu_*$ (which exist,
since $Q(x_*)$ is positive definite by our assumption; see, e.g. \cite[Thm.~1.1.4]{IoffeTihomirov} and
Lemma~\ref{lem:LICQvsPositiveDef}), the triplet $(x_*, \lambda_*, \mu_*)$ is a point of global minimum of
$\mathscr{L}(\cdot, r)$ by virtue of the fact that $\mathscr{L}(x_*, \lambda_*, \mu_*, r) = f(x_*) = f_*$. 

Suppose now that $(x_*, \lambda_*, \mu_*)$ is a point of global minimum of $\mathscr{L}(\cdot, r)$. Let us check that
$(x_*, \lambda_*, \mu_*)$ is a KKT-point and $x_*$ is a globally optimal solution of the problem $(\mathcal{P})$. 
Indeed, observe that the function
\begin{equation} \label{eq:AugmLagr_InequalityPart}
  \left\langle \mu, \max\left\{ g(x), -\frac{p(x, \mu)}{c} \mu \right\} \right\rangle
  + \frac{c}{2 p(x, \mu)} \left| \max\left\{ g(x), -\frac{1}{c} p(x, \mu) \mu \right\} \right|^2
\end{equation}
is nondecreasing in $c$. Consequently, if $F(x_*) \ne 0$, then
\begin{align*}
  \mathscr{L}(x_*, \lambda_*, &\mu_*, r) = f(x_*) 
  + \langle \lambda_*, F(x_*) \rangle  + \frac{r}{2} (1 + \| \lambda_* \|^2) \phi(\| F(x_*) \|^2)
  \\
  &+ \left\langle \mu_*, \max\left\{ g(x_*), -\frac{p(x_*, \mu_*)}{r} \mu_* \right\} \right\rangle
  \\
  &+ \frac{r}{2 p(x_*, \mu_*)} \left| \max\left\{ g(x_*), -\frac{p(x_*, \mu_*)}{r} \mu_* \right\} \right|^2
  + \eta(x_*, \lambda_*, \mu_*)
  \\
  &\ge \mathscr{L}(x_*, \lambda_*, \mu_*, c) + \frac{(r - c)}{2} (1 + \| \lambda_* \|^2) \phi(\|F(x_*)\|^2) 
  \\
  &\ge f_* + \frac{(r - c)}{2} (1 + \| \lambda_* \|^2) \phi(\|F(x_*)\|^2) > f_* = \inf_{(x, \lambda, \mu)}
  \mathscr{L}(x, \lambda, \mu, r).
\end{align*}
Therefore $F(x_*) = 0$. Arguing in a similar way and applying the fact that the function
\eqref{eq:AugmLagr_InequalityPart} is strictly increasing in $c$, if $\max\{ g(x), -\mu \} \ne 0$, one can easily
check that $\max\{ g(x_*), - \mu_* \} = 0$, which implies that $x_*$ is a feasible point of the problem $(\mathcal{P})$.
Hence one has 
\[
  f_* = \inf_{(x, \lambda, \mu)} \mathscr{L}(x, \lambda, \mu, r) = 
  \mathscr{L}(x_*, \lambda_*, \mu_*, r) = f(x_*) + \eta(x_*, \lambda_*, \mu_*),
\]
which implies that $f(x_*) = f_*$ and $\eta(x_*, \lambda_*, \mu_*) = 0$, since the function $\eta$ is nonnegative. Thus,
$x_*$ is a globally optimal solution of the problem $(\mathcal{P})$, and it remains to check that 
$(x_*, \lambda_*, \mu_*)$ is a KKT-point of the problem $(\mathcal{P})$.

Observe that for any $\lambda \in H$ and $\mu \in \mathbb{R}^m$ one has
\[  
  \eta(x_*, \lambda, \mu) = \frac{1}{2} \left\| \mathcal{E}(x_*) \mathcal{E}(x_*)^* 
  \left( \begin{smallmatrix} \lambda \\ \mu \end{smallmatrix} \right) 
  + \left( \begin{smallmatrix} DF(x_*)[\nabla f(x_*)] \\ \nabla g(x_*)[\nabla f(x_*)] \end{smallmatrix} \right)
  \right\|^2
\]
(see Corollary~\ref{crlr:LICQvsPositiveDef}). Therefore, $\eta(x_*, \lambda, \mu) = 0$ if and only if
\[
  \mathcal{E}(x_*) \mathcal{E}(x_*)^* \left( \begin{smallmatrix} \lambda \\ \mu \end{smallmatrix} \right) = 
  - \left( \begin{smallmatrix} DF(x_*)[\nabla f(x_*)] \\ \nabla g(x_*)[\nabla f(x_*)] \end{smallmatrix} \right).
\]
By Corollary~\ref{crlr:LICQvsPositiveDef} the operator $\mathcal{E}(x_*) \mathcal{E}(x_*)^*$ is invertible, which
implies that the equation above has a unique solution. Therefore, $\eta(x_*, \lambda, \mu) = 0$ if and only if 
$\lambda = \lambda_*$ and $\mu = \mu_*$.

As was noted above, from the fact that $Q(x_*)[\cdot]$ is positive definite it follows that there exists Lagrange
multipliers $\lambda_0 \in H$ and $\mu_0 \in \mathbb{R}^m$ such that $(x_*, \lambda_0, \mu_0)$ is a KKT-point. By the
definition of $\eta$ (see \eqref{eq:EtaDef}) one has $\eta(x_*, \lambda_0, \mu_0) = 0$, which implies that 
$\lambda_0 = \lambda_*$, $\mu_0 = \mu_*$, and $(x_*, \lambda_*, \mu_*)$ is a KKT-point.
\end{proof}

Let us obtain several types of sufficient conditions for the global exactness of the augmented Lagrangian 
$\mathscr{L}(x, \lambda, \mu, c)$. We start with, perhaps, the most general conditions based on the direct usage of 
Theorem~\ref{thrm:GradientEstimates}. For the sake of completeness, we will explicitly formulate all assumptions of the
following theorem, although most of them coincide with the assumptions of Theorem~\ref{thrm:GradientEstimates} with
$V = \Omega_r(f_* + \varepsilon)$.

\begin{theorem} \label{thrm:GlobalExactness}
Let the following assumptions be valid:
\begin{enumerate} 
\item{$f$, $F$, and $g_i$, $i \in M$, are twice continuously Fr\'{e}chet differentiable on $\Omega$, $\phi$ is
continuously differentiable on its effective domain, and $\phi'(0) > 0$;
\label{assump:GlobEx1}}

\item{the function $\mathscr{L}(\cdot, c)$ is weakly sequentially l.s.c. for all $c > 0$;}

\item{the set $\Omega_r(f_* + \varepsilon) = 
\{ x \in \Omega \mid f(x) + r( \| F(x) \|^2 + |\max\{ g(x), 0 \}|^2 ) \le f_* + \varepsilon \}$ 
is bounded for some $r > 0$ and $\varepsilon > 0$;
} 

\item{the functions $f$, $g_i$, $i \in M$, and $F$, as well as their first and second order Fr\'{e}chet derivatives, are
bounded on $\Omega_r(f_* + \varepsilon)$;
\label{assump:GlobEx4}}

\item{there exists $a > 0$ such that for all $x \in \Omega_r(f_* + \varepsilon)$ one has
\[
  Q(x)[\lambda, \mu] \ge a \big( \| \lambda \|^2 + |\mu|^2 \big)
  \quad \forall \lambda \in H, \: \mu \in \mathbb{R}^m.
\]
}

\vspace{-5mm}
\end{enumerate}
Then the augmented Lagrangian $\mathscr{L}(x, \lambda, \mu, c)$ is globally exact.\end{theorem}

\begin{proof}
By Theorems~\ref{thrm:SublevelBoundedness} and \ref{thrm:GradientEstimates} there exists $c_* > 0$ such that 
for all $c \ge c_*$ the set $S_c(f_*)$ is bounded and
\begin{equation} \label{eq:GradEstimateViaInfeasibility}
  \big\| \nabla \mathscr{L}(x, \lambda, \mu, c) \big\| 
  \ge \| F(x) \| + \left|\max\left\{ g(x), - \frac{1}{c} p(x, \mu) \mu \right\} \right|
\end{equation}
for all $(x, \lambda, \mu) \in S_c(f_*)$. 

Taking into account inequality \eqref{eq:WeakDuality} and the facts that $\mathscr{L}(\cdot, c)$ is weakly sequentially
l.s.c., and $X$ and $H$ are Hilbert spaces, one can conclude that for any $c \ge c_*$ the function 
$\mathscr{L}(\cdot, c)$ attains a global minimum at a point $(x(c), \lambda(c), \mu(c)) \in S_c(f_*)$. From the fact
that the augmented Lagrangian is Fr\'{e}chet differentiable on its effective domain by Proposition~\ref{prp:Derivatives}
it follows that $\nabla \mathscr{L}(x(c), \lambda(c), \mu(c), c) = 0$ for all $c \ge c_*$. Therefore, by
\eqref{eq:GradEstimateViaInfeasibility} the point $x(c)$ is feasible for the problem $(\mathcal{P})$ and 
\[
  \mathscr{L}(x(c), \lambda(c) \mu(c), c) = f(x(c)) + \eta(x(c), \lambda(c), \mu(c)) \ge f(x(c)) \ge f_*
\]
for any $c \ge c_*$. Hence bearing in mind \eqref{eq:WeakDuality} and Lemma~\ref{lem:GlobalExact_OptimalValue} one can
conclude that the augmented Lagrangian $\mathscr{L}(x, \lambda, \mu, c)$ is globally exact.
\end{proof}

The most restrictive assumption of the previous theorem is, of course, the assumption on the uniform positive
definiteness of the quadratic function $Q(x)[\cdot]$ on the set $\Omega_r(f_* + \varepsilon)$, which can be viewed as a
uniform nonlocal constraint qualification or an assumption on the nonlocal metric regularity of constraints. In various
particular cases, one can significantly relax this assumption and replace it with a local constraint qualification.
Here we present two simple results of this kind, merely to illustrate how one can use a particular structure/properties
of the problem under consideration to relax the assumptions of Theorem~\ref{thrm:GlobalExactness}.

Being inspired by the ideas of Zaslavsky \cite{Zaslavski,Zaslavski2009,Zaslavski2013} on the theory of exact penalty
functions, first we strengthen Theorem~\ref{thrm:GlobalExactness} with the use of the Palais-Smale condition
\cite{MawhinWillem}. To introduce a suitable version of this condition, denote by $a_{\max}(Q(x))$ the supremum of all
$a \ge 0$ such that $Q(x)[\cdot]$ is positive definite with constant $a$. As is easily seen,
\[
  a_{\max}(Q(x)) = \inf\Big\{ Q(x)[\lambda, \mu] \Bigm| \| (\lambda, \mu) \| = 1 \Big\},
\]
and in the finite dimensional case $a_{\max}(Q(x))$ is simply the smallest eigenvalue of the matrix of the quadratic
form $Q(x)[\cdot]$. Moreover, with the use of the equality above and the definition of $Q(x)[\cdot]$ one can check that 
the value $a_{\max}(Q(x))$ \textit{continuously} depends on $x$.

\begin{definition} \label{def:PalaisSmale}
One says that the constraints of the problem $(\mathcal{P})$ satisfy the Palais-Smale condition, if every bounded
sequence $\{ x_n \} \subset X$, such that 
\begin{enumerate}
\item{$\| F(x_n) \| + |\max\{ g(x_n), 0 \}| \to 0$ as $n \to \infty$,}

\item{$a_{\max}(Q(x_n)) \to 0$ as $n \to \infty$,}
\end{enumerate}
has a convergent subsequence.
\end{definition}

To understand how the above definition is connected with the traditional Palais-Smale condition, consider the simplest
case when there are no inequality constraints and $H = \mathbb{R}$. Then, as one can readily verify,
\[
  Q(x)[\lambda] = \frac{1}{2} |\nabla F(x)|^4 \lambda^2, \quad 
  a_{\max}(Q(x)) = \frac{1}{2} |\nabla F(x)|^4.
\]
Therefore, in this case the constraint $F(x) = 0$ satisfies the Palais-Smale condition from Def.~\ref{def:PalaisSmale},
if every bounded sequence $\{ x_n \}$, such that $F(x_n) \to 0$ and $\| \nabla F(x_n) \| \to 0$ as $n \to \infty$, has a
convergent subsequence. This is the well-known localized Palais-Smale condition \cite{MawhinWillem}. 

With the use of the Palais-Smale condition from Def.~\ref{def:PalaisSmale} we can significantly relax the nonlocal
constraint qualification from Theorem~\ref{thrm:GlobalExactness}.

\begin{theorem} \label{thrm:GlobalExactness_PalaisSmale}
Let Assumption~\ref{assump:GlobEx1}--\ref{assump:GlobEx4} of Theorem~\ref{thrm:GlobalExactness} be valid, and suppose
that the constraints of the problem $(\mathcal{P})$ satisfy the Palais-Smale condition and the function $Q(x)[\cdot]$ is
positive definite at every globally optimal solution of the problem $(\mathcal{P})$. Then the augmented Lagrangian
$\mathscr{L}(x, \lambda, \mu, c)$ is globally exact.
\end{theorem}

\begin{proof}
We divide the proof of the theorem into two parts. First we show that the sublevel set $S_c(f_*)$ is bounded for any
sufficiently large $c$, and then with the use of Theorem~\ref{thrm:GradientEstimates} and the Palais-Smale conditions
we will prove the global exactness of the augmented Lagrangian.

\textbf{Part~1.} Arguing by reductio ad absurdum, suppose that the set $S_c(f_*)$ is unbounded for any $c > 0$. Then
for any increasing unbounded sequence $\{ c_n \} \subset (0, + \infty)$ one can find 
$(x_n, \lambda_n, \mu_n) \in S_{c_n}(f_*)$, $n \in \mathbb{N}$, such that $\| x_n \| + \| \lambda_n \| + |\mu_n| \ge n$.

As was shown in the proof of Theorem~\ref{thrm:SublevelBoundedness}, there exists $\widehat{c} > 0$ such that for any
$c \ge \widehat{c}$ and $(x, \lambda, \mu) \in S_c(f_*)$ one has $x \in \Omega_r(f_* + \varepsilon)$. Therefore, the
sequence $\{ x_n \}$ is bounded and $\| \lambda_n \| + |\mu_n| \to + \infty$ as $n \to \infty$.

Suppose that there exist $a > 0$ and a subsequence $\{ x_{n_k} \}$ such that for all $k \in \mathbb{N}$ one has
$a_{\max}(Q(x_{n_k})) \ge a$ . Then replacing the sequence $\{ x_n \}$ with this subsequence and almost literally
repeating the proof of Theorem~\ref{thrm:SublevelBoundedness} one check that the condition 
$\| \lambda_n \| + |\mu_n| \to + \infty$ as $n \to \infty$ contradicts the assumption that
$(x_n, \lambda_n, \mu_n) \in S_c(f_*)$.

Thus, without loss of generality one can suppose that $a_{\max}(Q(x_n)) \to 0$ as $n \to \infty$. By
Lemma~\ref{lem:LowerEstimate} and the definition of $(x_n, \lambda_n, \mu_n)$ one has
\[
  f(x_n) + \frac{c_n}{2} \phi(\| F(x_n) \|^2) + \frac{c_n}{2 \psi(0)} \big| \max\{ g(x_n), 0 \} \big|^2
  \le f_* + \frac{1}{2 c_n \phi_0} + \frac{(1 + m)\psi(0)}{2c_n}.
\]
Therefore, as one can readily verify, one has
\begin{equation} \label{eq:PSLimitCond}
  \lim_{n \to \infty} \Big( \| F(x_n) \| + \big| \max\{ g(x_n), 0 \} \big| \Big) = 0, \quad
  \limsup_{n \to \infty} f(x_n) \le f_*.
\end{equation}
Hence by the Palais-Smale condition one can extract a subsequence $\{ x_{n_k} \}$ converging to some point
$x_*$, which is obviously feasible and satisfies the inequality $f(x_*) \le f_*$. Consequently, $x_*$ is a globally
optimal solution of the problem $(\mathcal{P})$, which by our assumption implies that $a_* := a_{\max}(Q(x_*)) > 0$. As
was noted above, the function $a_{\max}(Q(\cdot))$ is continuous. Therefore, $a_{\max}(Q(x_{n_k})) \ge a_* / 2$ for all
sufficiently large $k$, which contradicts the fact that $a_{\max}(Q(x_n)) \to 0$ as $n \to \infty$. Thus, there exists
$c_0 > 0$ such that the set $S_c(f_*)$ is bounded for all $c \ge c_0$.

\textbf{Part~2.} Bearing in mind inequality \eqref{eq:WeakDuality}, and the facts that $\mathscr{L}(\cdot, c)$ is weakly
sequentially l.s.c. and the set $S_c(f_*)$ is bounded for all $c \ge c_0$, one can conclude that the augmented
Lagrangian $\mathscr{L}(\cdot, c)$ attains a global minimum at a point $(x(c), \lambda(c), \mu(c)) \in S_c(f_*)$ for
any $c \ge c_0$. 

Choose an increasing unbounded sequence $\{ c_n \} \subset [c_0, + \infty)$, and denote
$\xi_n = (x_n, \lambda_n, \mu_n) = (x(c_n), \lambda(c_n), \mu(c_n))$. Observe that the sequence
$\{ (x_n, \lambda_n, \mu_n) \} \subset S_c(f_*)$ is bounded, since $S_c(f_*) \subseteq S_{c_0}(f_*)$ for any
$c \ge c_0$ due to the fact that the function $\mathscr{L}(x, \lambda, \mu, c)$ is nondecreasing in $c$.
Note that increasing $c_0$, if necessary, one can suppose that $\{ x_n \} \subset \Omega_r(f_* + \varepsilon)$ for all
$n \in \mathbb{N}$, thanks to Lemma~\ref{lem:LowerEstimate}.

Suppose, at first, that there exist $a > 0$ and a subsequence $\{ x_{n_k} \}$ such that 
$a_{\max}(Q(x_{n_k})) \ge a$ for all $k \in \mathbb{N}$. Then setting $V = \{ x_{n_k} \}$, 
$\Lambda = \{ (\lambda_{n_k}, \mu_{n_k}) \}$, $\gamma = f_*$, and $K = 1$ in Theorem~\ref{thrm:GradientEstimates},
one obtains that there exists $k_0 \in \mathbb{N}$ such that for any $k \ge k_0$ the following inequality holds true:
\begin{equation} \label{eq:PS_ALGradSeq}
  0 = \big\| \nabla \mathscr{L}(\xi_{n_k}, c_{n_k}) \big\| 
  \ge \| F(x_{n_k}) \| 
  + \left|\max\left\{ g(x_{n_k}), - \frac{p(x_{n_k}, \mu_{n_k})}{c_{n_k}} \mu_{n_k} \right\} \right|
\end{equation}
(the first equality follows from the fact that $\xi_{n_k} = (x_{n_k}, \lambda_{n_k}, \mu_{n_k})$ is a point of global
minimum of $\mathscr{L}(\cdot, c_{n_k})$ by definition). Hence arguing in the same way as in the proof of
Theorem~\ref{thrm:GlobalExactness}, one can conclude that the augmented Lagrangian is globally exact.

Thus, one can suppose that the sequence $\{ a_{\max}(Q(x_n)) \}$ does not contain a subsequence that is bounded away
from zero. Hence without loss of generality one can assume that $a_{\max}(Q(x_n)) \to 0$ as $n \to \infty$.

Applying the lower estimate of the augmented Lagrangian from Lemma~\ref{lem:LowerEstimate} and the fact that
$(x_n, \lambda_n, \mu_n) \in S_{c_n}(f_*)$, one can easily check that
\[
  \lim_{n \to \infty} \Big( \| F(x_n) \| + \big|\max\{ g(x_n), 0 \}\big| \Big) = 0, \quad
  \limsup_{n \to \infty} f(x_n) \le f_*.
\]
Consequently, by the Palais-Smale condition there exists a subsequence $\{ x_{n_k} \}$ converging to a point $x_*$,
which is obviously a globally optimal solution of the problem $(\mathcal{P})$. Therefore, by our assumption 
$Q(x_*)[\cdot]$ is positive definite, and due to the continuity of the function $a_{\max}(Q(\cdot))$ there exist 
$a_* > 0$ and $k_* \in \mathbb{N}$ such that $a_{\max}(Q(x_{n_k})) \ge a_*$ for all $k \ge k_*$, which contradicts our
assumption that $a_{\max}(Q(x_n)) \to 0$ as $n \to \infty$.
\end{proof}

Let us also consider another way one can relax the nonlocal constraint qualification from
Theorem~\ref{thrm:GlobalExactness} with the use of a particular structure of the constraints, which can be applied, e.g.
to variational problems with nonlinear constraints at the boundary of the domain. Namely, let $Y$ be a real Hilbert
space, and suppose that the constraints of the problem $(\mathcal{P})$ have the form
\begin{equation} \label{eq:CompactEmbConstrants}
  F(x) = F_0(\mathcal{A}x), \quad g(x) = g_0(\mathcal{A}x) \quad \forall x \in X,
\end{equation}
where $F_0 \colon Y \to H$ and $g_0 \colon Y \to \mathbb{R}^m$ are continuously differentiable nonlinear maps,
while $\mathcal{A} \colon X \to Y$ is a \textit{compact} linear operator. Thus, the constraints are defined via a
compact embedding of the space $X$ into another Hilbert space $Y$.

For any $y$ introduce the function
\begin{align*}
  Q_0(y)[\lambda, \mu] 
  &= \frac{1}{2} \Big\| D F_0(y)\Big[ D F_0(y)^*[\lambda] + \sum_{i = 1}^m \mu_i \nabla g_{0i}(y) \Big] \Big\|^2
  \\
  &+ \frac{1}{2} \Big| \nabla g_0(y) \Big( D F_0(y)^*[\lambda] + \sum_{i = 1}^m \mu_i \nabla g_{0i}(y) \Big) 
  + \diag(g_{0i}(y)^2) \mu \Big|^2,
\end{align*}
which is a modification of the function $Q(x)[\cdot]$ to the case of the constraints
\[
  F_0(y) = 0, \quad g_{0i}(y) \le 0, \quad i \in M.
\]
It is convenient to formulate sufficient conditions for the global exactness of the augmented Lagrangian for the
problem under consideration in terms of the function $Q_0(y)$.

\begin{theorem}
Let the following assumptions be valid:
\begin{enumerate} 
\item{$f$, $F_0$, and $g_0$ are twice continuously Fr\'{e}chet differentiable, $\phi$ is continuously differentiable on
its effective domain, and $\phi'(0) > 0$;}

\item{the functions $f$ and $\mathscr{L}(\cdot, c)$, $c > 0$, are weakly sequentially l.s.c.;}

\item{the set $\Omega_r(f_* + \varepsilon)$ is bounded for some $r > 0$ and $\varepsilon > 0$;} 

\item{the functions $f$, $g_i$, $i \in M$, and $F$, as well as their first and second order Fr\'{e}chet derivatives, are
bounded on $\Omega_r(f_* + \varepsilon)$;}

\item{the operator $\mathcal{A} \mathcal{A}^*$ is the identity map;}

\item{for any globally optimal solution $x_*$ of the problem $(\mathcal{P})$ the function $Q_0(y)[\cdot]$ is positive
definite at the point $y = \mathcal{A} x_*$.}
\end{enumerate}
Then the augmented Lagrangian $\mathscr{L}(x, \lambda, \mu, c)$ is globally exact.
\end{theorem}

\begin{proof}
We split the proof of this theorem into two parts, in precisely the same way as the proof of
Theorem~\ref{thrm:GlobalExactness_PalaisSmale}.

\textbf{Part~1.} Let us prove that under the assumptions of the theorem the sublevel set $S_c(f_*)$ is bounded for any
sufficiently large $c$. Arguing by reductio ad absrudum, suppose that this claim is false. Then, just like in the
proofs of Theorems~\ref{thrm:SublevelBoundedness} and \ref{thrm:GlobalExactness_PalaisSmale}, one can show that for any
$n \in \mathbb{N}$ there exists $(x_n, \lambda_n, \mu_n) \in S_n(f_*)$ such that the sequence $\{ x_n \}$ is bounded,
but $\| \lambda_n \| + |\mu_n| \to + \infty$ as $n \to \infty$.

From the facts that $X$ is a Hilbert space and the sequence $\{ x_n \}$ is bounded it follows that one can extract a
subsequence $\{ x_{n_k} \}$ that weakly converges to some point $x_*$. Since the operator $\mathcal{A}$ is compact,
the sequence $\{ \mathcal{A} x_{n_k} \}$ converges to $\mathcal{A} x_*$ in the norm topology. 

With the use of Lemma~\ref{lem:LowerEstimate} one can readily verify that 
\[
  \lim_{n \to \infty} \Big( \| F(x_n) \| + \big| \max\{ g(x_n), 0 \} \big| \Big) = 0, \quad
  \limsup_{n \to \infty} f(x_n) \le f_*.
\]
Hence taking into account the facts that $f$ is weakly sequentially l.s.c., and $\mathcal{A} x_{n_k}$ converges to
$\mathcal{A} x_*$ in the norm topology one can conclude that $x_*$ is a globally optimal solution of the problem
$(\mathcal{P})$. Thus, $Q_0(\mathcal{A} x_*)$ is positive definite. Therefore, there exists $k_0 \in \mathbb{N}$ such
that
\[
  a_{\max}(Q_0(\mathcal{A} x_{n_k})) \ge \frac{a_*}{2} \quad \forall k \ge k_0, \quad
  a_* := a_{\max}(Q_0(\mathcal{A} x_*)) > 0,
\]
due to the facts that the function $a_{\max}(Q_0(\cdot))$ is continuous in the norm topology and $\mathcal{A} x_{n_k}$
strongly converges to $\mathcal{A} x_*$.

From the definition of $Q(x)[\cdot]$ (see equality \eqref{eq:QuadFormDef} on page~\pageref{eq:QuadFormDef}), equalities
\eqref{eq:CompactEmbConstrants}, and the fact that $\mathcal{A} \mathcal{A}^*$ is the identity map it follows that
\begin{multline*}
  Q(x)[\lambda, \mu]
  = \frac{1}{2} \bigg\| D F_0(\mathcal{A} x) \Big( \mathcal{A} \mathcal{A}^*
  \Big[ D F_0(\mathcal{A} x)^*[\lambda] + \sum_{i = 1}^m \mu_i \nabla g_{0i}(\mathcal{A} x) \Big]\Big) \bigg\|^2
  \\
  + \frac{1}{2} \bigg| \nabla g_0(\mathcal{A} x) \Big[ \mathcal{A} \mathcal{A}^*
  \Big( D F_0(\mathcal{A} x)^*[\lambda] + \sum_{i = 1}^m \mu_i \nabla g_{0i}(\mathcal{A} x) \Big) \Big]
  + \diag(g_{0i}(\mathcal{A} x_*)^2) \mu \bigg|^2
  \\
  = Q_0(\mathcal{A} x)[\lambda, \mu]
\end{multline*}
for all $x \in X$, $\lambda \in H$ and $\mu \in \mathbb{R}^m$. Consequently, $a_{\max}(Q(x_{n_k})) \ge a_* / 2$ for all 
$k \ge k_0$. Hence applying Lemma~\ref{lem:LowerEstimate} one obtains that
\begin{align*}
  \mathscr{L}(x_{n_k}, \lambda_{n_k}, \mu_{n_k}, n_k) 
  &\ge f(x_{n_k}) - \frac{1}{2 n_k \phi'(0)} - \frac{(1 + m)\psi(0)}{2n_k} + \eta(x_{n_k}, \lambda_{n_k}, \mu_{n_k})
  \\
  &\ge f(x_{n_k}) - \frac{1}{2 n_k \phi'(0)} - \frac{(1 + m)\psi(0)}{2n_k} 
  + \frac{a_*}{2} \big\| (\lambda_{n_k}, \mu_{n_k}) \big\|^2 
  \\
  &- \| Q_{1, \lambda}(x_{n_k}) \| \| \lambda_{n_k} \|
  - \big| Q_{1, \mu}(x_{n_k}) \big| |\mu_{n_k}| + Q_0(x_{n_k})
\end{align*}
for all $k \ge k_0$ (here we used the same notation as in the proof of Theorem~\ref{thrm:SublevelBoundedness}). As was
noted multiple times above, $x_n \in \Omega_r (f_* + \varepsilon)$ for any sufficient large $n$. Consequently, 
the quantities $\| Q_{1, \lambda}(x_{n_k}) \|$, $|Q_{1, \mu}(x_{n_k})|$, and $Q_0(x_{n_k})$ are bounded due to our
assumption on the boundedness of all functions and their derivatives on $\Omega_r (f_* + \varepsilon)$. Hence one
gets that $\mathscr{L}(x_{n_k}, \lambda_{n_k}, \mu_{n_k}, n_k) \to + \infty$ as $k \to \infty$, which contradicts
the fact that by definition $\mathscr{L}(x_n, \lambda_n, \mu_n, n) \le f_*$ for all $n \in \mathbb{N}$. Thus, there
exists $c_0 > 0$ such that the sublevel set $S_c(f_*)$ is bounded for all $c \ge c_0$.

\textbf{Part~2.} Let us now prove the global exactness of the augmented Lagrangian. Choose an increasing unbounded
sequence $\{ c_n \} \subseteq [c_0, + \infty)$. From the fact that the sublevel set $S_c(f_*)$ is bounded for all 
$c \ge c_0$ and the augmented Lagrangian is weakly sequentially l.s.c. it follows that for any $n \in \mathbb{N}$
the function $\mathscr{L}(\cdot, c_n)$ attains a global minimum at a point $(x_n, \lambda_n, \mu_n) \in S_{c_n}(f_*)$.

Note that the sequence $\{ (x_n, \lambda_n, \mu_n) \}$ is bounded, due to the fact that 
$S_c(f_*) \subseteq S_{c_0}(f_*)$ for all $c \ge c_0$ . Therefore, replacing, if necessary, this sequence with a
subsequence, one can suppose that the sequence $\{ x_n \}$ weakly converges to some point $x_*$. The corresponding
sequence $\{ \mathcal{A} x_n \}$ strongly converges to $\mathcal{A} x_*$ due to the fact that the operator $\mathcal{A}$
is compact. Hence with the use of the lower estimate from Lemma~\ref{lem:LowerEstimate} and the fact that 
$(x_n, \lambda_n, \mu_n) \in S_{c_n}(f_*)$ one can easily verify that $x_*$ is a globally optimal solution of the
problem $(\mathcal{P})$. Consequently, the function $Q_0(\mathcal{A} x_*)[\cdot] = Q(x_*)[\cdot]$ is positive definite,
and one can find $a > 0$ and $n_0 \in \mathbb{N}$ such that 
\[
  a_{\max}(Q(x_n)) = a_{\max}(Q_0(\mathcal{A} x_n)) \ge a \quad \forall n \ge n_0,
\]
due to the facts that the map $a_{\max}(Q_0(\cdot))$ is contiuous and $\mathcal{A} x_n$ strongly converges to
$\mathcal{A} x_*$.

Now, applying Theorem~\ref{thrm:GradientEstimates} with $V = \{ x_n \}_{n \ge n_0}$ (one can obviously suppose that 
$\{ x_n \}_{n \ge n_0} \subset \Omega_r(f_* + \varepsilon)$), $\Lambda = \{ (\lambda_n, \mu_n) \}_{n \ge n_0}$, 
$\gamma = f_*$, and $K = 1$ one obtains that there exists $N \ge n_0$ such that for any $n \ge N$ one has
\[
  0 = \big\| \nabla \mathscr{L}(x_n, \lambda_n, \mu_n, c_n) \big\| 
  \ge \| F(x_n) \| + \left|\max\left\{ g(x_n), - \frac{p(x_n, \mu_n)}{c_n} \mu_n \right\} \right|.
\]
Therefore, for any $n \ge N$ the point $x_n$ is feasible for the problem $(\mathcal{P})$ and
\[
  f_* \ge \mathscr{L}(x_n, \lambda_n, \mu_n, c_n) = f(x_n) + \eta(x_n, \lambda_n, \mu_n) \ge f(x_n) \ge f_*.
\]
Hence by Lemma~\ref{lem:GlobalExact_OptimalValue} the augmented Lagrangian is globally exact.
\end{proof}

\begin{remark}
Let us point out one important particular case in which the assumption of the previous theorem that 
$\mathcal{A} \mathcal{A}^*$ is the identity map holds true. Namely, let $X$ be the Sobolev space 
$H^1([a, b]; \mathbb{R}^d)$ of vector-valued functions $x \colon [a, b] \to \mathbb{R}^d$ endowed with the inner
product
\[
  \langle x, y \rangle_X = \langle x(a), y(a) \rangle + \langle x(a) + x(b), y(a) + y(b) \rangle
  + (b - a) \int_a^b \langle \dot{x}(t), \dot{y}(t) \rangle \, dt
\]
and the corresponding norm, which is equivalent to the standard norm on $H^1([a, b]; \mathbb{R}^d)$. Suppose also that
$Y$ is the space $\mathbb{R}^d \times \mathbb{R}^d$ endowed with the inner product
\[
  \langle (x_1, x_2), (y_1, y_2) \rangle_Y = 3 \langle x_1, y_1 \rangle + 2 \langle x_2, y_2 \rangle
  \quad \forall (x_1, x_2), (y_1, y_2) \in Y. 
\]
Let $\mathcal{A} x = (x(a), x(b))$. In this case, the constraints $F(x) = 0$ and $g(x) \le 0$ restrict the values of 
the function $x$ at the boundary points $t = a$ and $t = b$. As is easy seen, one has
\[
  \Big( \mathcal{A}^* (y_1, y_2) \Big)(t) = y_1 + (y_2 - y_1) \frac{t - a}{b - a}
  \quad \forall t \in [a, b], \: (y_1, y_2) \in Y,
\]
since for all $x \in X$ and $y = (y_1, y_2) \in Y$ the following equalities hold true:
\begin{align*}
  \langle \mathcal{A}^* y, x \rangle_X &= \langle y_1, x(a) \rangle + \langle y_1 + y_2, x(a) + x(b) \rangle
  + \int_a^b \langle y_2 - y_1, \dot{x}(t) \rangle \, dt
  \\
  &= \langle y_1, x(a) \rangle + \langle y_1 + y_2, x(a) + x(b) \rangle + \langle y_2 - y_1, x(b) - x(a) \rangle
  \\
  &= 3 \langle y_1, x(a) \rangle + 2 \langle y_2, x(b) \rangle = \langle y, \mathcal{A} x \rangle_Y.
\end{align*}
It remains to note that in this case $\mathcal{A} \mathcal{A}^*$ is indeed the identity map. Moreover, the operator
$\mathcal{A}$ is obviously compact, which allows one to apply the previous theorem to corresponding problems.
\end{remark}

\subsection{Complete exactness}

In many cases, optimization methods can find only points of local minimum or even only stationary (critical) points of a
nonconvex function. Therefore, apart from global exactness, it is important to have conditions ensuring that not only
points of global minimum of the augmented Lagrangian $\mathscr{L}(x, \lambda, \mu, c)$ correspond to points of global
minimum of the original problem $(\mathcal{P})$, but also points of local minimum/stationary points of the augmented
Lagrangian correspond to points of local minimum/KKT-points of the problem $(\mathcal{P})$. Under such conditions the
problem of unconstrained minimisation of the augmented Lagrangian $\mathscr{L}(x, \lambda, \mu, c)$ is, in a sense,
completely equivalent to the original problem $(\mathcal{P})$. In this case it is natural to call $\mathscr{L}(x,
\lambda, \mu, c)$ \textit{completely exact} (cf. completely exact penalty functions in
\cite{DolgopolikFominyh,Dolgopolik2020}).

The following theorem contains natural sufficient conditions for the complete exactness of the augmented Lagrangian
on the sublevel set $S_c(\gamma)$. The question of whether the complete exactness of this augmented Lagrangian on 
the entire space $X \times H \times \mathbb{R}^d$ can be proved under some additional assumptions remains an
interesting open problem.

\begin{theorem} \label{thrm:CompleteExactness}
Let the assumptions of Theorem~\ref{thrm:GradientEstimates} be satisfied for $V = \Omega_r(\gamma + \varepsilon)$ with
some $r > 0$, $\varepsilon > 0$, and $\gamma > f_*$, and suppose that the augmented Lagrangian $\mathscr{L}(\cdot, c)$
is weakly sequentially l.s.c. for all $c > 0$. Then there exists $c_* > 0$ such that for all $c \ge c_*$ the augmented
Lagrangian $\mathscr{L}(x, \lambda, \mu, c)$ is completely exact on the set $S_c(\gamma)$ in the following sense:
\begin{enumerate}
\item{the optimal values of the problem $(\mathcal{P})$ and problem \eqref{eq:AuxiliaryProblem} coincide;
}

\item{$(x_*, \lambda_*, \mu_*)$ is point of global minimum of $\mathscr{L}(x, \lambda, \mu, c)$ if and only if 
$x_*$ is a globally optimal solution of the problem $(\mathcal{P})$ and $(x_*, \lambda_*, \mu_*)$ is a KKT-point 
of this problem;
}

\item{$(x_*, \lambda_*, \mu_*) \in S_c(\gamma)$ is a stationary point of $\mathscr{L}(x, \lambda, \mu, c)$ if and only
if $(x_*, \lambda_*, \mu_*)$ is a KKT-point of the problem $(\mathcal{P})$ and $f(x_*) \le \gamma$;
}

\item{if $(x_*, \lambda_*, \mu_*) \in S_c(\gamma)$ is a point of local minimum of $\mathscr{L}(x, \lambda, \mu, c)$, 
then $x_*$ is a locally optimal solution of the problem $(\mathcal{P})$, $f(x_*) \le \gamma$, and 
$(x_*, \lambda_*, \mu_*)$ is a KKT-point of this problem.
}
\end{enumerate}
\end{theorem}

\begin{proof}
Note that by the definition of global exactness the validity of the first two statements of the theorem follows 
directly from Theorem~\ref{thrm:GlobalExactness}. Let us prove the last two statements of the theorem. We prove 
the statement about stationary points first, since its proof is simpler than the proof of the statement on locally
optimal solutions.

\textbf{Part~1.} By Theorem~\ref{thrm:GradientEstimates} there exists $c_* > 0$ such that for all $c \ge c_*$ and 
$(x, \lambda, \mu) \in S_c(\gamma)$ the lower estimate of the gradient \eqref{eq:GradEstim_InfesasibilityMes} holds
true. Consequently, for any $c \ge c_*$ and any stationary point $\xi_* = (x_*, \lambda_*, \mu_*) \in S_c(\gamma)$ of 
$\mathscr{L}(x, \lambda, \mu, c)$ one has
\begin{equation} \label{eq:FeasibilityAtKKTPoint}
  F(x_*) = 0, \quad \max\left\{ g(x_*), - \frac{p(x_*, \mu_*)}{c} \mu_* \right\} = 0,
\end{equation}
which implies that 
\[
  f(x_*) \le f(x_*) + \eta(x_*, \lambda_*, \mu_*) = \mathscr{L}(x_*, \lambda_*, \mu_*, c) \le \gamma,
\]
and $x_* \in \Omega_r(\gamma + \varepsilon)$ for any $c > 0$. Hence the quadratic function $Q(x_*)[\cdot]$ is positive
definite by our assumption. 

Observe that from \eqref{eq:FeasibilityAtKKTPoint}, the equality $\nabla \mathscr{L}(x_*, \lambda_*, \mu_*, c) = 0$, 
and Proposition~\ref{prp:Derivatives} (see also \eqref{eq:MultipliersDerivExpr}) it follows that
\[
  0 = \begin{pmatrix} 
    \nabla_{\lambda} \mathscr{L}(\xi_*, c) \\ 
    \nabla_{\mu} \mathscr{L}(\xi_*, c)
  \end{pmatrix}
  = \mathcal{E}(x_*) \mathcal{E}(x_*)^*
  \begin{pmatrix}
    D F(x_*)[\nabla_x L(\xi_*)] \\ 
    \nabla g(x_*) \nabla_x L(\xi_*) + \diag(g_i(x_*)^2) \mu_*
  \end{pmatrix}.
\]
By Corollary~\ref{crlr:LICQvsPositiveDef} the operator $\mathcal{E}(x_*) \mathcal{E}(x_*)^*$ is invertible, which
yields
\[
  D F(x_*)[\nabla_x L(\xi_*)] = 0, \quad 
  \nabla g(x_*) \nabla_x L(\xi_*) + \diag(g_i(x_*)^2) \mu_* = 0.
\]
Hence applying \eqref{eq:FeasibilityAtKKTPoint} and Proposition~\ref{prp:Derivatives} once again one gets that
\[
  0 = \nabla_x \mathscr{L}(x_*, \lambda_*, \mu_*, c) = \nabla_x L(x_*, \lambda_*, \mu_*),
\] 
that is, $(x_*, \lambda_*, \mu_*)$ is a KKT-point of the problem $(\mathcal{P})$.

Conversely, let $(x_*, \lambda_*, \mu_*)$ be a KKT-point of the problem $(\mathcal{P})$ such that $f(x_*) \le \gamma$. 
Then by definition 
\[
  \nabla_x L(x_*, \lambda_*, \mu_*) = 0, \quad F(x_*) = 0, \quad \max\{ g(x_*), \mu_* \} = 0.
\]
Therefore, $\mathscr{L}(x_*, \lambda_*, \mu_* c) = f(x_*) \le \gamma$, i.e. $(x_*, \lambda_*, \mu_*) \in S_c(\gamma)$
for any $c > 0$. Furthermore, with the use of Proposition~\ref{prp:Derivatives} one obtains that for any $c > 0$ the
equality $\nabla \mathscr{L} (x_*, \lambda_*, \mu_*, c) = 0$ holds true, that is, $(x_*, \lambda_*, \mu_*)$ is a
stationary point of $\mathscr{L}(x, \lambda, \mu, c)$ for all $c > 0$.

\textbf{Part~2.} Let $c_* > 0$ be as above. Let us now show that for all $c \ge c_*$ and for any point of local minimum 
$(x_*, \lambda_*, \mu_*) \in S_c(\gamma)$ of $\mathscr{L}(x, \lambda, \mu, c)$ the point $x_*$ is a locally optimal
solution of the problem $(\mathcal{P})$. The fact that $(x_*, \lambda_*, \mu_*)$ is a KKT-point follows directly from
the previous part of the proof.

Indeed, fix any $c \ge c_*$ and a point of local minimum $(x_*, \lambda_*, \mu_*) \in S_c(\gamma)$ of 
$\mathscr{L}(x, \lambda, \mu, c)$. Note that $\nabla \mathscr{L}(x_*, \lambda_*, \mu_*, c) = 0$ by the necessary 
optimality condition. Therefore, equalities \eqref{eq:FeasibilityAtKKTPoint} hold true due to our choice of $c_*$. 
With the use of these equalities one gets that
\[
  f(x_*) \le f(x_*) + \eta(x_*, \lambda_*, \mu_*) = \mathscr{L}(x_*, \lambda_*, \mu_*, c) \le \gamma
\]
and $x_* \in \Omega_t(\gamma)$ for any $t > 0$. Hence the quadratic form $Q(x_*)[\cdot]$ is positive definite by
our assumption. Furthermore, from the previous part of the proof it follows that $(x_*, \lambda_*, \mu_*)$ is a
KKT-point of the problem $(\mathcal{P})$, which implies that $\eta(x_*, \lambda_*, \mu_*) = 0$ and
$\mathscr{L}(x_*, \lambda_*, \mu_*, c) = f(x_*)$.

By the definition of local minimum there exist neighbourhoods $\mathcal{U}_x$ of $x_*$, 
$\mathcal{U}_{\lambda}$ of $\lambda_*$, and $\mathcal{U}_{\mu}$ of $\mu_*$ such that
\[
  f(x_*) = \mathscr{L}(x_*, \lambda_*, \mu_*, c) \le \mathscr{L}(x, \lambda, \mu, c)
  \quad \forall (x, \lambda, \mu) \in \mathcal{U} 
  := \mathcal{U}_x \times \mathcal{U}_{\lambda} \times \mathcal{U}_{\mu}.
\]
Note that
\[
  \left\langle \mu, \max\left\{ g(x), -\frac{p(x, \mu)}{c} \mu \right\} \right\rangle
  + \frac{c}{2 p(x, \mu)} \left| \max\left\{ g(x), -\frac{p(x, \mu)}{c} \mu \right\} \right|^2 \le 0
\]
for any $x$ such that $g(x) \le 0$ (see the proof of \cite[Prop.~3.1, part~(b)]{DiPilloLucidi2001}). 
Therefore by the definition of augmented Lagrangian \eqref{eq:ExactAugmLagr}, for any 
$(x, \lambda, \mu) \in \mathcal{U}$ such that $F(x) = 0$ and $g(x) \le 0$ one has
\begin{equation} \label{eq:AugmLagr_LocMin_FeasPoints}
  f(x_*) \le \mathscr{L}(x, \lambda, \mu, c) \le f(x) + \eta(x, \lambda, \mu).
\end{equation}
As was noted above, the quadratic form $Q(x_*)[\cdot]$ is positive definite, which by 
Corollary~\ref{crlr:LICQvsPositiveDef} implies that the linear operator $\mathcal{E}(x_*) \mathcal{E}(x_*)^*$ is
invertible. It is easily seen that under our assumptions the operator $\mathcal{E}(x) \mathcal{E}(x)^*$ continuously
depends on $x$. Hence taking into account the fact that the set of invertible operators is open and the inversion is
continuous in the uniform operator topology (see, e.g. \cite[Thm.~10.12]{Rudin}), one obtains that there exists 
a neighbourhood $\mathcal{V}_x \subseteq \mathcal{U}_x$ of $x_*$ such that for any $x \in \mathcal{V}_x$ the operator 
$\mathcal{E}(x) \mathcal{E}(x)^*$ is invertible and the corresponding inverse operator continuously depends on $x$.

For any $x \in \mathcal{V}_x$ define $\lambda(x) \in H$ and $\mu(x) \in \mathbb{R}^m$ as a unique solution of 
the following equation:
\begin{equation} \label{eq:VariableMultipliersDef}
  \mathcal{E}(x) \mathcal{E}(x)^* \left( \begin{smallmatrix} \lambda \\ \mu \end{smallmatrix} \right) = 
  - \left( \begin{smallmatrix} DF(x)[\nabla f(x)] \\ \nabla g(x)[\nabla f(x)] \end{smallmatrix} \right).
\end{equation}
By the definition of $\mathcal{E}(x)$ (see Corollary~\ref{crlr:LICQvsPositiveDef}) one has
\begin{align*}
  D F(x)[\nabla_x L(x, \lambda(x), \mu(x))] &= 0,
  \\
  \nabla g(x) \nabla_x L(x, \lambda(x), \mu(x)) + \diag(g_i(x)^2) \mu(x) &= 0
\end{align*}
for all $x \in \mathcal{V}_x$, which implies that $\eta(x, \lambda(x), \mu(x)) = 0$ and, thanks to the uniqueness of
solution of \eqref{eq:VariableMultipliersDef}, $\lambda(x_*) = \lambda_*$ and $\mu(x_*) = \mu_*$. Furthermore,
$\lambda(x)$ and $\mu(x)$ continuously depend on $x$, since the inverse operator of $\mathcal{E}(x) \mathcal{E}(x)^*$
and the right-hand side of \eqref{eq:VariableMultipliersDef} continuously depend on 
$x \in \mathcal{V}_x$. Consequently, replacing, if necessary, the neighbourhood $\mathcal{V}_x$ with a smaller one, we
can suppose that $(x, \lambda(x), \mu(x)) \in \mathcal{U}$ for all $x \in \mathcal{V}_x$. Hence with the use of
\eqref{eq:AugmLagr_LocMin_FeasPoints} one obtains that
\[
  f(x_*) \le \mathscr{L}(x, \lambda(x), \mu(x), c) \le f(x) + \eta(x, \lambda(x), \mu(x)) = f(x)
\]
for any $x \in \mathcal{V}_x$ that is feasible for the problem $(\mathcal{P})$. In other words, $x_*$ is a locally
optimal solution of this problem.
\end{proof}

\subsection{Local exactness}

Although Theorem~\ref{thrm:CompleteExactness} somewhat completely describes an intimate relation between optimal
solutions/KKT-point of the problem $(\mathcal{P})$ and minimisers/statio\-na\-ry points of the augmented Lagrangian
$\mathscr{L}(x, \lambda, \mu, c)$ for any sufficiently large $c > 0$, it does not tell one whether locally optimal
solutions of the problem $(\mathcal{P})$ correspond to the points of local minimum of the augmented Lagrangian. It
is possible that some KKT-points $(x_*, \lambda_*, \mu_*)$, corresponding to \textit{locally} optimal solutions $x_*$
of the problem $(\mathcal{P})$, are only stationary points of $\mathscr{L}(x, \lambda, \mu, c)$, but not its points of
local minimum. The aim of this seciton is to provide simple sufficient conditions for such KKT-points to be points of
local minimum of the augmented Lagrangian for any $c > 0$ large enough. 

Let $(x_*, \lambda_*, \mu_*)$ be a KKT-point of the problem $(\mathcal{P})$, and $f$, $F$, and $g$ be twice Fr\'{e}chet
differentiable at $x_*$. One says that \textit{the second order sufficient optimality conditions} hold true at $x_*$, if
there exists $\rho > 0$ such that
\begin{equation} \label{eq:SecondOrderOptCond}
  D^2_{xx} L(x_*, \lambda_*, \mu_*)[z, z] \ge \rho \| z \|^2 \quad \forall z \in \mathscr{C}(x_*),
\end{equation}
where
\[
  \mathscr{C}(x_*) = \Big\{ z \in X \mid D F(x_*)[z] = 0, 
  \: \langle \nabla g_i(x_*), z \rangle = 0, \: i \in M(x_*) \Big\}
\]
is \textit{the critical cone} at the point $x_*$. We say that the strict complementarity condition is satisfied for the
KKT-point $(x_*, \lambda_*, \mu_*)$, if $(\mu_*)_i > 0$ for any index $i \in M(x_*)$.

\begin{theorem}
Let $(x_*, \lambda_*, \mu_*)$ be a KKT-point of the problem $(\mathcal{P})$ satisfying the strict complementarity
condition, the functions $f$, $F$, and $g$ be twice Fr\'{e}chet differentiable at $x_*$, $\phi$ be differentiable at
zero, and $\phi'(0) > 0$. Suppose also that $Q(x_*)[\cdot]$ is positive definite and the second order sufficient
optimality conditions hold true at $x_*$. Then there exist $c_* > 0$ and $\theta > 0$ such that for any $c \ge c_*$ the
triplet $(x_*, \lambda_*, \mu_*)$ is point of isolated local minimum of $\mathscr{L}(\cdot, c)$ and 
\[
  \mathscr{L}(x, \lambda, \mu, c) \ge \mathscr{L}(x_*, \lambda_*, \mu_*, c) 
  + \theta \Big( \| x - x_* \|^2 + \| \lambda - \lambda_* \|^2 + |\mu - \mu_*|^2 \Big).
\]
for any $(x, \lambda, \mu)$ in a neighbourhood of $(x_*, \lambda_*, \mu_*)$.
\end{theorem}

\begin{proof}
Fix some $\theta > 0$ and denote $\xi = (x, \lambda, \mu)$, $\xi_* = (x_*, \lambda_*, \mu_*)$, and
\[
  \omega_c(\xi) = \mathscr{L}(\xi, c) 
  - \theta \Big( \| x - x_* \|^2 + \| \lambda - \lambda_* \|^2 + |\mu - \mu_*|^2 \Big).
\]
Our aim is to compute a second order expansion of the function $\omega_c$ in a neighbourhood of $\xi_*$ and utilise it
to prove the theorem.

Firstly, note that from Proposition~\ref{prp:Derivatives} and the fact that $(x_*, \lambda_*, \mu_*)$ is a KKT-point it
follows that $\nabla \mathscr{L}(x_*, \lambda_*, \mu_*, c) = 0$ for all $c > 0$. Hence $\nabla \omega_c(\xi_*) = 0$.

For any $i \in I$ denote
\[
  G_i(\xi, c) = \mu_i \max\left\{ g_i(x), - \frac{p(x, \mu)}{c} \mu_i \right\}
  + \frac{c}{2 p(x, \mu)} \left\{ g_i(x), - \frac{p(x, \mu)}{c} \mu_i \right\}^2.
\]
For any $i \in M_1 = M \setminus M(x_*) = \{ i \in M \colon g_i(x_*) < 0 \}$ one has $(\mu_*)_i = 0$ and
\[
  G_i(\xi_*, c) = 0, \quad G_i(\xi, c) = - \frac{p(x, \mu)}{2c} \mu_i^2
\]
in a neighbourhood of $\xi_*$, which implies that
\[
  G_i(\xi_* + \Delta \xi, c) = - \frac{\psi(0)}{2 c} \Delta \mu_i^2 + o(\| \Delta \xi \|^2).
\]
In turn, for any $i \in M_2 := M(x_*)$ one has $(\mu_*)_i > 0$, thanks to the strict complementarity condition, and
\[
  G_i(\xi_*, c) = 0, \quad G_i(\xi, c) = \mu_i g_i(x) + \frac{c}{2 p(x, \mu)} g_i(x)^2
\]
in a neighbourhood of $\xi_*$, which yields the expansion:
\begin{align*}
  G_i(\xi_* + \Delta \xi, c) &= 
  (\mu_*)_i \Big( \langle \nabla g_i(x_*), \Delta x \rangle + \frac{1}{2} D^2 g_i(x_*)[\Delta x, \Delta x] \Big)
  \\
  &+ \Delta \mu_i \langle \nabla g_i(x_*), \Delta x \rangle 
  + \frac{c}{2 \psi(0)} \langle \nabla g_i(x_*), \Delta x \rangle^2 + o(\| \Delta \xi \|^2).
\end{align*}
Hence taking into account the definition of the augmented Lagrangian one gets that the function $\omega_c(\cdot)$ admits
the following second order expansion in a neighbourhood of $\xi_*$:
\begin{equation} \label{eq:SecondOrderExpans}
  \omega_c(\xi_* + \Delta \xi) - \omega_c(\xi_*) = H_c(\Delta x) + Q(x_*)[\Delta \lambda, \Delta \mu]
  + R_c(\Delta \xi) - \theta \| \Delta \xi \|^2 + o(\| \Delta \xi \|^2).
\end{equation}
Here
\begin{equation} \label{eq:SecondOrderXPart}
\begin{split}
  H_c(\Delta x) &= \frac{1}{2} D^2_{xx} L(x_*, \lambda_*, \mu_*)[\Delta x, \Delta x]
  \\
  &+ \frac{c}{2} \big( 1 + \| \lambda_* \|^2 ) \phi'(0) \big\| D F(x)[\Delta x] \big\|^2
  + \frac{c}{2 \psi(0)} \sum_{i \in M_2} \langle \nabla g_i(x_*), \Delta x \rangle^2
\end{split}
\end{equation}
and
\begin{align*}
  R_c(\Delta \xi) &= \langle \Delta \lambda, D F(x_*)[\Delta x] \rangle 
  - \sum_{i \in M_1} \frac{\psi(0)}{2 c} \Delta \mu_i^2
  + \sum_{i \in M_2} \Delta \mu_i \langle \nabla g_i(x_*), \Delta x \rangle
  \\
  &+ \frac{1}{2} \Big\| D F(x_*)\Big[ P[\Delta x] \Big] \Big\|^2
  + \frac{1}{2} \sum_{i = 1}^m 
  \Big\langle \nabla g_i(x_*), P[\Delta x] \Big \rangle^2
  \\
  &+ \Big\langle D F(x_*)\Big[ P[\Delta x] \Big], 
  D F(x_*)\Big[ Z[\Delta \lambda, \Delta \mu] \Big] \Big\rangle
  \\
  &+ \sum_{i = 1}^m
  \Big\langle \nabla g_i(x_*), P[\Delta x] \Big \rangle
  \Big( \langle \nabla g_i(x_*), Z[\Delta \lambda, \Delta \mu] \rangle
  + g_i(x)^2 \Delta \mu_i \Big),
\end{align*}
where 
\[
  P[\Delta x] = D_x(\nabla_x L(\xi_*))[\Delta x], \quad
  Z[\Delta \lambda, \Delta \mu] = D F(x_*)^*[\Delta \lambda] + \sum_{i = 1}^m \Delta \mu_i \nabla g_i(x_*).
\]
Let us estimate the function $R_c(\cdot)$ from below. Fix some $\varepsilon \in (0, 1)$. Applying
the inequality
\[
  \langle x, y \rangle \ge - \| x \| \| y \| 
  \ge - \frac{(1 + \varepsilon)}{2} \| x \|^2 - \frac{1}{2(1 + \varepsilon)} \| y \|^2
  \quad \forall x, y \in X
\]
to the last two terms of $R_c(\cdot)$ and taking into account the definition of $Q(x)[\cdot]$ (see
\eqref{eq:QuadFormDef}), one obtains that
\begin{multline*}
  R_c(\Delta \xi) \ge \langle \Delta \lambda, D F(x_*)[\Delta x] \rangle 
  - \sum_{i \in M_1} \frac{\psi(0)}{2 c} \Delta \mu_i^2
  + \sum_{i \in M_2} \Delta \mu_i \langle \nabla g_i(x_*), \Delta x \rangle
  \\
  - \frac{\varepsilon}{2} \Big\| D F(x_*)\Big[ P[\Delta x] \Big] \Big\|^2
  - \frac{\varepsilon}{2} \sum_{i = 1}^m \Big\langle \nabla g_i(x_*), P[\Delta x] \Big \rangle^2
  - \frac{1}{1 + \varepsilon} Q(x_*)[\Delta \lambda, \Delta \mu].
\end{multline*}
Clearly, there exists $K > 0$ such that
\[
  \frac{1}{2} \Big\| D F(x_*)\Big[ P[\Delta x] \Big] \Big\|^2
  + \frac{1}{2} \sum_{i = 1}^m \Big\langle \nabla g_i(x_*), P[\Delta x] \Big \rangle^2
  \le K \| \Delta x \|^2.
\]
Denote $a = a_{\max}(Q(x_*))$. Applying the inequality
\[
  \langle x, y \rangle \ge - \frac{\varepsilon a}{4} \| x \|^2 - \frac{1}{\varepsilon a} \| y \|^2 
  \quad \forall x, y \in X
\]
to the first and third terms of $R_c(\cdot)$, one finally gets that for any $c > 2 \psi(0) / \varepsilon a$ the
following inequality holds true:
\begin{align*}
  R_c(\Delta \xi) \ge &- \frac{1}{\varepsilon a} \| D F(x_*)[\Delta x] \|^2
  - \frac{1}{\varepsilon a} \sum_{i \in M_2} \langle \nabla g_i(x_*), \Delta x \rangle^2 
  - K \varepsilon \| \Delta x \|^2
  \\
  &- \frac{\varepsilon a}{4} \| (\Delta \lambda, \Delta \mu) \|^2 
  - \frac{1}{1 + \varepsilon} Q(x_*)[\Delta \lambda, \Delta \mu].
\end{align*}
Hence taking into account \eqref{eq:SecondOrderExpans} and the definition of $H_c(\Delta x)$ one obtains that
\begin{align*}
  \omega_c(\xi_* + \Delta \xi) - \omega_c(\xi_*) &\ge H_{c - c_0}(\Delta x)
  + \frac{\varepsilon a}{4} \| (\Delta \lambda, \Delta \mu) \|^2 
  \\
  &- K \varepsilon \| \Delta x \|^2 - \theta \| \Delta \xi \|^2 + o(\| \Delta \xi \|^2),
\end{align*}
for any $\varepsilon \in (0, 1)$ and $c > c_0 := 2 \max\{ \psi(0), 1 / \phi'(0) \} / \varepsilon a$. 

Let us check that there exist $c_* > 0$ and $\beta > 0$ such that
\begin{equation} \label{eq:XPartPositiveDef}
  H_c(\Delta x) \ge \beta \| \Delta x \|^2 \quad \forall \Delta x \in X, \: c > c_*.
\end{equation}
Then choosing any $0 < \varepsilon < \beta / 3 K$ and $0 < \theta < \min\{ \varepsilon a / 8, \beta / 3 \}$ one
obtains that
\[
  \omega_c(\xi_* + \Delta \xi) - \omega_c(\xi_*) 
  \ge \min\left\{ \frac{\varepsilon a}{8}, \frac{\beta}{3} \right\} \| \Delta \xi \|^2 + o(\| \Delta \xi \|^2),
\]
for any $c > c_* + c_0$. Hence, as one can easily check, $\omega_c(\xi) \ge \omega_c(\xi_*)$ for any $\xi$ from 
a sufficiently small neighbourhood of $\xi_*$, which implies the required result.

Thus, it remains to prove inequality \eqref{eq:XPartPositiveDef}. To this end, introduce the linear operator
$\mathcal{T} \colon X \to H \times \mathbb{R}^{m(x_*)}$, defined as
\[
  \mathcal{T} z = \big\{ D F(x_*)[z] \big\} \times \prod_{i \in M(x_*)} \{ \langle \nabla g_i(x_*), z \rangle \}
  \quad \forall z \in X.
\] 
Note that the kernel of this operator coincides with the critical cone $\mathscr{C}(x_*)$. For any $z \in X$, below we
denote by $z_1$ the orthogonal projection of $z$ onto $\mathscr{C}(x_*)$ and by $z_2$ the orthogonal projection of $z$
onto the orthogonal complement of $\mathscr{C}(x_*)$. Then $z = z_1 + z_2$ for any $z \in X$.

Let $\Theta > 0$ be such that
\[
  \Big| D^2_{xx} L(\xi_*)[x, y] \Big| \le \Theta \| x \| \| y \| \quad \forall x, y \in X.
\]
Then with the use of the second order sufficient optimality conditions \eqref{eq:SecondOrderOptCond} one gets that
\begin{multline*}
  \frac{1}{2} D^2_{xx} L(\xi_*)[z, z] = \frac{1}{2} D^2_{xx} L(\xi_*)[z_1, z_1]
  + D^2_{xx} L(\xi_*)[z_1, z_2] + \frac{1}{2} D^2_{xx} L(\xi_*)[z_2, z_2]
  \\
  \ge \rho \| z_1 \|^2 - \Theta \| z_1 \| \| z_2 \| - \Theta \| z_2 \|^2
  \ge \frac{\rho}{2} \| z_1 \|^2 - \left( \Theta + \frac{\Theta^2}{2 \rho} \right) \| z_2 \|^2,
\end{multline*}
for any $z \in X$.

By Lemma~\ref{lem:LICQvsPositiveDef} the operator $\mathcal{T}$ is surjective due to our assumption on the positive
definiteness of $Q(x_*)[\cdot]$. Consequently, by the open mapping theorem there exists $\tau > 0$ such that
\[
  \| \mathcal{T} z \| \ge \tau \| z_2 \| \quad \forall z \in X
\]
(see, e.g. \cite{BorweinDontchev})). Hence taking into account the definition of $H_c$ (see \eqref{eq:SecondOrderXPart})
one obtains that
\begin{align*}
  H_c(z) &= \frac{1}{2} D^2_{xx} L(x_*, \lambda_*, \mu_*)[z, z]
  + \frac{c}{2} \big( 1 + \| \lambda_* \|^2 ) \phi'(0) \big\| D F(x)[z] \big\|^2
  \\
  &+ \frac{c}{2 \psi(0)} \sum_{i \in M(x_*)} \langle \nabla g_i(x_*), z \rangle^2 
  \\
  &\ge \frac{\rho}{2} \| z_1 \|^2 - \left( \Theta + \frac{\Theta^2}{2 \rho} \right) \| z_2 \|^2
  + \frac{c}{2} \min\left\{ \phi'(0), \frac{1}{\psi_0} \right\} \| \mathcal{T} z \|^2
  \\
  &\ge \frac{\rho}{2} \| z_1 \|^2 
  + \left( \frac{c \tau}{2} \min\left\{ \phi'(0), \frac{1}{\psi(0)} \right\} - \Theta - \frac{\Theta^2}{2 \rho} \right) 
  \| z_2 \|^2
  \ge \frac{\rho}{2} \| z \|^2
\end{align*}
for any $z \in X$ and 
\[
  c > \frac{\rho^2 + 2 \Theta \rho + \Theta^2}{\rho \tau \min\{ \phi'(0), 1 / \psi(0) \}},
\]
which completes the proof of the theorem.
\end{proof}

\section{Conclusions}

In this paper, we developed a general theory of exact augmented Lagrangians for constrained optimization problems in
Hilbert spaces with inequality and nonlinear operator equality constraints. The core result of this theory is the lower
estimate of the gradient of the augmented Lagrangian via the infeasibility measure from
Theorem~\ref{thrm:GradientEstimates}, which allows one to obtain several types of sufficient conditions for the global
or complete exactness of the augmented Lagrangian. These conditions ensure that local/global minimisers or critical
points of the augmented Lagrangian correspond to locally/globally optimal solutions or KKT-points of the constrained
optimization problem. Main results of the paper are obtained with the use of a nonlocal constraint qualification, which
is reduced to LICQ in this finite dimensional case, and is closely related to assumptions on nonlocal metric regularity
of constraints.

Various applications of the theoretical results from this paper to constrained variational problems, problems with
PDE constraints, and optimal control problems, as well as several numerical examples, will be presented in 
the second part of our study.

\bibliographystyle{abbrv}  
\bibliography{ExAugmLagr_bibl}

\end{document}